\documentclass[11pt,twoside,reqno]{amsart}
\usepackage[T1]{fontenc}
\usepackage[latin9]{inputenc}
\usepackage{amsmath} 
\usepackage{amsthm} 
\usepackage{amssymb}
\usepackage{esint}
\usepackage{enumerate}
\usepackage{graphicx}

\makeatletter
\newcommand{\hide}[1]{}

\DeclareMathOperator{\id}{id}

\DeclareMathOperator{\diam}{diam}
\DeclareMathOperator{\supp}{supp}

\DeclareMathOperator{\bdim}{dim_\textrm{B}}

\DeclareMathOperator{\adim}{dim_\textrm{A}}
\DeclareMathOperator{\lydim}{dim_\textrm{L}}

\def\conv{\mbox{\LARGE{$.$}}}

\def\conv{\mbox{\LARGE{$.$}}}



\numberwithin{equation}{section}

\usepackage{color}
\usepackage{dsfont}
\usepackage{comment}

\DeclareMathOperator{\edim}{\dim_e}
\newcommand{\AAA}{\mathbb{A}}
\newcommand{\R}{\mathbb{R}}
\newcommand{\N}{\mathbb{N}}
\newcommand{\ii}{\mathbf{i}}
\newcommand{\jj}{\mathbf{j}}
\newcommand{\uu}{\mathbf{u}}
\newcommand{\II}{\mathbf{I}}
\newcommand{\JJ}{\mathbf{J}}
\newcommand{\RP}{{\mathbb{RP}^1}}

\newcommand{\UU}{\mathbf{U}}

\newcommand{\QQ}{\mathcal{Q}}
\newcommand{\DD}{\mathcal{Q}}
\newcommand{\PP}{\mathbb{P}}
\newcommand{\EE}{\mathbb{E}}
\newcommand{\m}{\nu}
\newcommand{\ep}{\varepsilon}

\theoremstyle{plain}
\newtheorem{theorem}{Theorem}[section]
\newtheorem*{thm*}{Theorem}

\newtheorem{corollary}[theorem]{Corollary}
\newtheorem{proposition}[theorem]{Proposition}
\newtheorem{lemma}[theorem]{Lemma}

\theoremstyle{remark}
\newtheorem{remark}[theorem]{Remark}

\newtheorem*{claim}{Claim}

\theoremstyle{definition}
\newtheorem{definition}[theorem]{Definition}

\begin{document}

\title[Hausdorff dimension of planar self-affine sets and measures]{Hausdorff dimension of planar self-affine sets and measures}

\author{Bal\'azs B\'ar\'any}
\address[Bal\'azs B\'ar\'any]
        {Budapest University of Technology and Economics \\
        MTA-BME Stochastics Research Group \\
        P.O.\ Box 91 \\
        1521 Budapest \\
        Hungary \&\\
        Einstein Institute of Mathematics \\
Edmond J. Safra Campus (Givat Ram) \\
The Hebrew University \\
Jerusalem 91904 \\
Israel}
\email{balubsheep@gmail.com}

\author{Michael Hochman}

\author{Ariel Rapaport}
\address[Michael Hochman, Ariel Rapaport]
        {Einstein Institute of Mathematics \\
Edmond J. Safra Campus (Givat Ram) \\
The Hebrew University \\
Jerusalem 91904 \\
Israel}
\email[Michael Hochman]{mhochman@math.huji.ac.il}

\email[Ariel Rapaport]{ariel.rapaport@math.huji.ac.il}

\date{\today}

\begin{abstract}
Let $X=\bigcup\varphi_{i}X$   be a strongly separated self-affine set in $\R^2$ (or one satisfying the strong open set condition). Under mild non-compactness and irreducibility assumptions on the matrix parts of the $\varphi_{i}$, we prove that $\dim X$   is equal to the affinity dimension, and similarly for self-affine measures and the Lyapunov dimension. The proof is via analysis of the dimension of the orthogonal projections of the measures, and relies on additive combinatorics methods.
\end{abstract}

\maketitle

\section{Introduction}

Computing the dimension of self-affine sets -- attractors of systems
of affine contractions of $\mathbb{R}^{d}$ -- is one of the major
open problems in fractal geometry. The corresponding problem for self-similar sets and measures is far better understood, and, when a mild separation condition
is present, completely solved \cite{Hutchinson1981,Hochman2014}. The affine case is more delicate and
less understood. Falconer established a general upper bound on the
dimension in terms of the so-called affinity dimension \cite{Falconer1988}, 
and many authors have obtained matching lower bounds in special cases,
e.g. for typical self-affine sets in which some of the maps are randomized
\cite{Falconer1988,Solomyak1998,JordanPollicottSimon2007}, or for
special classes, such as those satisfying bunching conditions \cite{HueterLalley1995,KaenmakiShmerkin2009}. It has also long been realized that the dimension is closely related to the dimension of projections. Falconer solved the problem for some self-affine sets in $\R^d$ assuming a uniform lower bound on the Lebesgue measure of projections to $d-1$-dimensional subspaces \cite{Falconer1992}, but this hypothesis rarely holds. Recently, analysis of projections was used to solve the problem under the assumption that the
Furstenberg measure of the associated matrix random walk is sufficiently
large \cite{FalconerKempton2016,Rapaport2016,Barany2015}.

All these results provide
strong evidence that the affinity dimension is typically the right
one, but one should note that the dimension is not always equal to
the affinity dimension, as demonstrated by carpet-like fractals \cite{McMullen1984, Bedford1989-box-dimension-of-repellers, GatzourasLalley1992-dim-of-certain-self-affine-fractals, Baransky2007-dim-of-limit-sets-planar-constrictions, Fraser2012-dim-of-box-like-fractals-in-R2}.
So the challenge is to find mild, general conditions under which the
affinity dimension is achieved. The purpose of this paper is to do
just that in the case of separated self-affine measures in
the plane which exhibit some mild non-compactness and irreducibility of their matrix
parts.

Throughout the paper we use the following notation, see also Section \ref{sec:notation-and-background} for further definitions. $\Phi=\{\varphi_{i}\}_{i\in\Lambda}$ be a finite system
of affine contractions of $\mathbb{R}^{2}$, with $\varphi_{i}(x)=A_{i}x+b_{i}$
for some $A_{i}\in GL_{2}(\mathbb{R})$ and $b_{i}\in\mathbb{R}^{2}$. We say that $A_i$ is the linear part of $\varphi_i$. We also write $\overline{A}_i=A_i/\sqrt{|\det A_i|} \in GL_2(\mathbb{R})$ and say that it is the normalized linear part of $\varphi_i$ (these are the matrices normalized to have determinant $\pm 1$). Let $X$ denote the attractor of $\Phi$ and write $\dim X$ and $\bdim X$ for its Hausdorff and box dimensions, and $\adim X$ for its affinity dimension; we do not define the latter because we will not
work with it directly, for a definition see \cite{Falconer1988}. The system $\Phi$ satisfies the strong open set condition (SOSC) if there is a bounded open set $U$ with $U\cap X\neq\emptyset$ such that $\varphi_i U\subseteq U$ for all $i\in\Lambda$ and the images $\varphi_i U$ are pairwise disjoint. This is satisfied, in particular, under strong separation, i.e. when the images $\varphi_iX$ are pairwise disjoint. Recall that a group
of matrices is totally irreducible if it does not preserve any finite
union of non-trivial linear spaces. 
\begin{theorem}
\label{thm:main-sets}Let $X=\bigcup\varphi_{i}X$ be a self affine set in $\mathbb{R}^{2}$ satisfying the strong open set condition, and suppose that the normalized linear
parts of $\varphi_{i}$ generate a non-compact and totally irreducible
group in $GL_{2}(\mathbb{R})$. Then $\dim X=\bdim X=\adim X$.
\end{theorem}

The theorem is proved, as is often the case, by analysis of self-affine
measures. Let $p=(p_{i})_{i\in\Lambda}$ be a strictly positive probability vector,
and let $\mu=\sum p_{i}\varphi_{i}\mu$ be the corresponding self-affine
measure for $\Phi$. Let $H(p)$ denote the entropy of $p$, and let
$\lambda_{2}\leq\lambda_{1}<0$ denote the Lyapunov exponent associated
to the i.i.d. random matrix product with $A_{i}$ chosen with probability
$p_{i}$. The Lyapunov dimension of $\mu$ (or, rather, of $\Phi$
and $p$) is the analogue of the affinity dimension for self-affine
sets, and is defined by 
\[
\lydim\mu=\left\{ \begin{array}{cc}
\frac{H(p)}{|\lambda_{1}|} & \mbox{if \ensuremath{\frac{H(p)}{|\lambda_{1}|}\leq1}}\\
1+\frac{H(p)-|\lambda_{1}|}{|\lambda_{2}|} & \mbox{otherwise}
\end{array}\right. .
\]
The pointwise dimension of $\mu$, denoted $\dim\mu$, is known to
exist, and always satisfies $\dim\mu\leq\lydim\mu$ \cite{JordanPollicottSimon2007}, but in general
there may be a strict inequality.
\begin{theorem}
\label{thm:main-measures}Let $\mu=\sum p_{i}\varphi_{i}\mu$ be a self-affine measure in $\mathbb{R}^{2}$  satisfying the strong open set condition, and suppose
that the normalized linear parts of $\varphi_{i}$ generate a non-compact
and totally irreducible subgroup of $GL_{2}(\mathbb{R})$. Then $\dim\mu=\lydim\mu$.
\end{theorem}
Theorem \ref{thm:main-sets} follows from Theorem \ref{thm:main-measures}
by the work of Morris-Shmerkin \cite{MorrisShmerkin2016},
who proved that under the hypotheses of Theorem \ref{thm:main-sets}, one can find an IFS $\Phi'\subseteq\bigcup_{n=1}^\infty \Phi^n$ which satisfies the same hypotheses, and which supports a self-affine measure whose Lyapunov dimension is arbitrarily close to the affinity dimension of the original system.
In order to show this approximation property, Morris and Shmerkin used the
K\"aenm\"aki measure, which is in our setup the unique ergodic shift
invariant measure, for which $\lydim\mu=\adim X$, see \cite{kaenmaki,Kaenmaki2004}.
However, the K\"aenm\"aki measure is far from being Bernoulli measure in
general. Thus, the question still remains open\footnote{Our methods actually apply to the images under the standard symbolic coding map of quasi-product measures, i.e. those which are equivalent, up to a shift, on every cylinder set, and with Radon-Nykodim derivative bounded away from $0$ and infinity. This is because in all entropy computations in the proofs, measures which are equivalent in this sense differ in entropy (w.r.t. \textit{any} partition)  by an addivie constant depending only on the bounds on the RN-derivatives; so one can just fix a representative and work with it. When the matrix parts of the IFS preserve a (multi) cone, the Kaenmaki measure is known to be quasi-Bernoulli. Thus, in this case, our methods give an invariant measure of maximal dimension. We thank Pablo Shmerkin for this remark.}  whether there is an ergodic
measure $\mu$ for which $\dim\mu=\dim X$.

The proof of Theorem \ref{thm:main-measures} rests on analysis of
``typical'' projections of $\mu$ to lines. Let $\eta^*$ be a measure
on the space of lines in $\mathbb{R}^{2}$ satisfying $\eta^*=\sum p_{i}\cdot A_{i}^*\eta^*$. 
Under the assumptions of the theorem, there is a unique such measure,
called the Furstenberg measure, which may also be described as the
distribution of the directions of the expanding Oseledets subspace
of the random matrix product in which each matrix is chosen equal
to $A_{i}^*$ with probability $p_{i}$ (see  Section \ref{sub:furstenberg}). For a 1-dimension subspace $W\leq\mathbb{R}^{2}$
write $\pi_{W}:\mathbb{R}^{2}\rightarrow W$ for the orthogonal projection
to $W$. In order to prove Theorem \ref{thm:main-measures}, it suffices,
by the Ledrappier-Young formula for self-affine measures \cite{BaranyKaenmaki2015},
to show that 
\begin{equation}
\dim\pi_{v}\mu=\min\{1,\dim\mu\}\label{eq:maximal-image}
\end{equation}
for $\eta^*$-a.e. $v$. We show that, in fact, (\ref{eq:maximal-image})
holds for all $v$ outside a set of dimension $0$. Since the hypothesis
of the theorem implies $\dim\eta^*>0$, it follows that (\ref{eq:maximal-image})
holds $\eta^*$-a.e. We note that a-postiori, it follows that (\ref{eq:maximal-image})
holds for all $v$ with at most one exception, and in fact, no exceptions if $\{\overline{A}_i\}$ satisfy the hypotheses of the theorem. 

This strategy is similar to other recent works on the topic \cite{Barany2015,FalconerKempton2016,Rapaport2016},
but those required $\eta^*$ to have dimension large with respect to
$\dim\mu$ of $\adim\mu$, at which point one can invoke classical
projection theorems of Marstrand, Falconer and others in order to
prove (\ref{eq:maximal-image}). But in general, the dimension of $\eta^*$ is unrelated to the dimension of $\mu$ and may be very small compared to it.

Our method, in contrast, is related
to the recent works on self-similar measures with overlap. Of course,
$\pi_{W}\mu$ is neither self-similar nor even self-affine, but it
nevertheless decomposes at all scales into statistically related measures
of comparable diameter. More precisely, for every $n$, we can choose
a set $\Phi_{n}\subseteq\Phi^{*}$ such that each $\varphi_\ii$, $\ii\in\Phi_{n}$,
contracts by $q^{-n+O(1)}$, and such that $\mu=\sum_{\ii\in\Phi_{n}}p_{\ii}\cdot\varphi_{\ii}\mu$,
where for a sequence $\ii=i_{1}\ldots i_{k}$ we have written $p_{\ii}=p_{i_{1}}\ldots p_{i_{k}}$
and $\varphi_{\ii}=\varphi_{i_{1}}\ldots\varphi_{i_{k}}$. Thus $\pi_{W}\mu=\sum_{\ii\in\Phi_{n}}p_{\ii}\cdot\pi_{W}\varphi_\ii\mu$,
which can be written, in turn, as $\pi_{W}\mu=\theta_{n}\conv\mu$,
where $\theta_{n}=\sum_{\ii\in\Phi_{n}}p_{\ii}\cdot\pi_{W}\circ\varphi_{\ii}$
is a measure supported on the space of linear maps $\mathbb{R}^{2}\rightarrow\mathbb{R}$,
and $\theta_{n}\conv\mu$ denotes the push-forward of $\theta_{n}\times\mu$
by the action map $(\pi,x)\mapsto\pi(x)$. The hypotheses of the theorems
imply that, after rescaling affine maps to be orthogonal projections, $\theta_{n}$ approaches the distribution of the random projection $\pi_{V}$,
$V\sim\eta^{*}$; and, for most choices of $W$, if $\dim\pi_{W}\mu$
does not satisfy (\ref{eq:maximal-image}), then $\theta_{n}$ will
have positive entropy, in a suitable sense. We have
arrived now at the point where we must analyze the entropy of convolutions.
This can be done by adapting methods developed in \cite{Hochman2014,Hochman2015},
whereby we replace small pieces of the non-linear convolution $\theta_{n}\conv\mu$
by a bona-fide convolution of measures on $\mathbb{R}$ and use results
on entropy growth under convolution to show that this results in more
dimension, or entropy, than there should be.

The property (\ref{eq:maximal-image}) requires less than strong separation. Let $D(\cdot,\cdot)$ denote the metric on the affine group of $\R^2$ induced from the operator norm when affine maps are embedded as linear maps of $\R^{3}$ (alternatively fix a left-invariant Riemannian metric on the affine group).
We say that a system
of affine maps $\{\varphi_{i}\}_{i\in\Lambda}$ satisfies exponential
separation, if there exists a constant $c>0$ such that, for every
$n$ and every pair of distinct sequences $\ii,\jj\in\Lambda^n$, the corresponding compositions satisfy
\[
D(\varphi_{\ii},\varphi_{\jj})>c^{n}.
\]
The strong open set condition (and also the weaker open set condition (OSC)) implies this property.
\begin{theorem}
\label{thm:one-dimensional-projections}Let $\mu=\sum p_{i}\cdot\varphi_{i}\mu_{i}$
be a self-affine measure in $\mathbb{R}^{2}$ such that $\{\varphi_{i}\}$
has exponential separation. Suppose that the normalized linear parts
of $\varphi_{i}$ generate a totally irreducible, non-compact subgroup
of $GL_{2}(\mathbb{R})$. Then for $\eta^*$-a.e. $V\in\RP$, we have \[
\dim\pi_{V}\mu=\min\{1,\dim\mu\}.
\]
\end{theorem}
This begs the question of whether exponential separation is enough
for Theorems \ref{thm:main-sets} and \ref{thm:main-measures},
instead of SOSC. The answer is yes, but not easily. The
reason is that in general, by the Ledrappier-Young formula for self-affine meausres \cite{BaranyKaenmaki2015}. equation (\ref{eq:maximal-image}) ensures only that 
\begin{equation}
\dim\mu=\left\{ \begin{array}{cc}
\frac{H(p)-\Delta}{|\lambda_{1}|} & \mbox{if \ensuremath{\frac{H(p)-\Delta}{|\lambda_{1}|}\leq1}}\\
1+\frac{H(p)-|\lambda_{1}|-\Delta}{|\lambda_{2}|} & \mbox{otherwise}
\end{array}\right. .\label{eq:LY-with-overlaps}
\end{equation}
where $\Delta$ is an entropy quantity arising from the amount of
overlap between cylinders (the entropy on the fibers of the projection
from the symbolic coding). The SOSC ensures that $\Delta=0$, and removing this assumption would require us to show directly that if $\dim\mu<\lydim\mu$ then
$\Delta=0$, which requires non-trivial new tools (this is one reason why we cannot replace the SOSC with the OSC in  Theorem \ref{thm:main-sets}: it does not, so far as we know, imply $\Delta=0$. One also must exclude degenerate situations in which the OSC is satisfied but the attractor is a single point).

The other obvious
question raised by our work is the extent to which the results hold
in $\mathbb{R}^{d}$. We will return to both these matters in a separate paper. 

Finally, we note that the conditions on the normalized linear parts of
$\varphi_{i}$ in the theorems above are there in order to ensure
uniqueness of the Furstenberg measure and that it have positive dimension.
The proof actually requires slightly less: we want a stationary measure
$\eta=\sum p_{i}\cdot A_{i}\eta$ to exist on the space of lines,
and that it should attract the associated random walk on lines (driven
by the distribution $\nu=\sum p_{i}\cdot\delta_{A_{i}}\in\mathcal{P}(GL_{2}(\mathbb{R}))$)
when started from any initial line outside a finite set of exceptions. Thus, our methods can handle the case of a non-compact but reducible system, which  is just a system whose linear parts are upper (or lower) triangular matrices in a suitable basis. This is an  important class of self-affine systems, which has been studied by Falconer and Miao~\cite{falcmiao}, Kirat and Kocyigit~\cite{kirat} and B\'ar\'any, Rams and Simon~\cite{brs}. For the  statements and proof see  Section \ref{sub:triangular-case}.

\subsection{Organization}
Section \ref{sec:notation-and-background} sets up definitions and notation, and some background (much of the notation is not standard, and should be read). Section \ref{sec:entropy-porosity} establishes entropy porosity of the projections of self-affine measures. Section \ref{sec:entropy-growth-under-convolutions} proves the version of the inverse theorem for entropy growth needed in this paper. Section \ref{sec:Separation} establishes the necessary exponential separation of projections of cylinders. Section \ref{sec:main-theorem} puts the pieces together and proves Theorem \ref{thm:one-dimensional-projections}. The last section presents some generalizations and applications to dynamical systems.

\subsection{Acknowledgment} This work was supported by ERC grant 306494. B.B. acknowledges support from the grants NKFI PD123970, OTKA K123782, and the J\'anos Bolyai Research Scholarship of the Hungarian Academy of Sciences.

\section{\label{sec:notation-and-background}Setup}
 
We refer the reader to \cite{Mattila1995} for basic background on dimension theory. We write $\mathcal{P}(X)$ for the space of Borel probability measures on a Borel space $X$. We rely on the standard notations $O(f(t)),o(f(t)),\Theta(f(t))$ for asymptotic behavior of functions and sequences.

\subsection{\label{sub:self-affine-measures}Self-affine sets and measures}

Throughout the paper, $\Phi=\{\varphi_i(x)=A_ix+a_i\}_{i\in\Lambda}$ is a system of affine contractions satisfying the strong open set condition, as in the introduction, and $X$ the associated attractor, defined by the relation \[
  X = \bigcup_{i\in\Lambda}\varphi_i(X).
\]
We also fix a strictly positive probability vector $p=(p_i)_{i\in\Lambda}$, and let $\mu$ denote the associated  self-affine measure, defined by the relation \[
  \mu = \sum_{i\in\Lambda}p_i \cdot \varphi_i\mu.
\]
We write $\Lambda^*$ for the set of all finite words over $\Lambda$. For a word $\ii=(i_0,\ldots,i_n)\in\Lambda^*$, let
\[
A_{\ii}=A_{i_0}\cdots A_{i_n},
\]
and similarly write $\varphi_\ii=\varphi_{i_0}\ldots\varphi_{i_n}$, $p_\ii=p_{i_0}\ldots p_{i_m}$, $A^*_{\ii}=A_{i_n}^*\cdots A_{i_0}^*$, etc. We define the ``projection'', or coding map, $\Pi:\Lambda^\N\to X$, by \[
  \Pi(\ii)=\lim_{n\to\infty}\varphi_{i_1\ldots,i_n}(0),
\]
where the limit exists by contraction. We write \[
  \nu=p^\N
\]
for the product measure on $\Lambda^\N$ with marginal $p$, so that $\mu=\Pi\nu$.

\subsection{Dilations, translation, projections}

The operations of dilation and translation in $\R^k$ we denote by $S_c$ and $T_a$ respectively, i.e., for $c\in\R$  we write $S_{c}(x)=c\cdot x$, and for $a\in\R^k$ we write $T_{a}(x)=x+a$.

We shall work extensively with orthogonal projections from $\R^2$ to subspaces $V\leq\R^2$. However, we shall want to choose coordinates on the image $V$, identifying it with $\R$, and turning the orthogonal projection into an affine map $\R^2\to\R$.  Choosing coordinates amounts to choosing, for each $V\in\RP$, a unit vector $u=u(V)\in V$. We choose $u(V)$ to be the unit vector such that $\langle u,\binom{1}{0}\rangle>0$. This leaves $u(V)$ undefined only when $V$ is the $y$-axis, and in this case we set $u=(0,1)$.

We can now associate to each $V\in\RP$ the ``orthogonal projection'' $\pi_V$ given by\[
\pi_V(x)=\langle x,u(V)\rangle.
\]
This gives a measurable (and mostly continuous) embedding of $\RP$ into the set of norm-$1$ affine maps $\R^2\to\R$.

For a linear map $T:\R^k\to \R^m$ and subspace $V\subseteq \R^k$, let $\|T|V\|$ denote the norm of the restricted map $T|_V$. Then for any affine  $\varphi:\R^2\to\R^2$ with $\varphi(x)=Ax+t$, it is easy to check that
\begin{equation}\label{eq:affine-map-to-projection}
\pi_V(\varphi(x))=(\pm1)\|A^*|V\|\cdot\pi_{A^*V}(x)+\pi_V(t),
\end{equation}
where the sign is chosen so that $\pm u(A^*V)=A^*u(V)/\|A^*u(V)\|$. Consequently,
\[
\|\pi_V\circ A_\ii\|=\|A_{i_1}^*|V\|\cdots\|A_{i_n}^*|A_{i_{n-1}}^*\cdots A_{i_0}^*V\|=\|A_{\ii}^*|V\|.
\]
(the equality of the left- and right-hand sides holds in $\R^d$ for any $V\subseteq\R^d$, but the middle expression relies on $V$ being $1$-dimensional).

\subsection{Affine maps and an invariant metric}\label{sub:affine-maps-and-invariant-metric}

Let $\AAA_{k,m}$ denote the space of affine maps $\mathbb{R}^{k}\mapsto\mathbb{R}^{m}$ of full rank.
For $\psi\in \AAA_{k,m}$, let $\|\psi\|$ denote the (operator) norm of the linear part of $\psi$, which is given explicitly by \[
\|\psi\|=\sup_{0\neq x\in\R^k}\frac{\|\psi(x)-\psi(0)\|}{\|x\|}.
\]
Note that the ``projections'' $\pi_V$ of the previous section form a bounded subset of $\AAA_{2,1}$.

The set $\AAA_{2,1}$ is a manifold, and we may parametrize it explicitly as follows. Denote the unit circle by
\[
S^1=\{y\in\mathbb{R}^{2}\::\:|y|=1\},
\]
and associate to $(t,u,a)\in\R\times S^1\times\R$ the map $x\mapsto e^t \langle x,u\rangle +a$. This is a bijection and provides a smooth structure to $A_{2,1}$.

The group $\AAA_{1,1}$ of affine maps of the line act naturally on $\AAA_{2,1}$ by post-composition: for $\varphi\in \AAA_{1,1}$ and $\psi\in \AAA_{2,1}$, let \[
\varphi\psi = \varphi \circ\psi.
\]
This is clearly a smooth action.

\begin{lemma}
There exists a metric $d$ on $\AAA_{2,1}$ which is $\AAA_{1,1}$-invariant, i.e.
\begin{equation}\label{eq:invariant-metric}
d(\phi\psi_{1},\phi\psi_{2})=d(\psi_{1},\psi_{2})\;\;\;\text{ for every }\psi_{1},\psi_{2}\in \AAA_{2,1}\text{ and }\phi\in \AAA_{1,1}\:.
\end{equation}
\end{lemma}

\begin{proof}

  The action of $\AAA_{1,1}$ on $\AAA_{2,1}$ is free (for each $\psi\in \AAA_{2,1}$, the map $\AAA_{1,1}\to \AAA_{2,1}$, $\varphi\to \varphi\psi$, is injective). The set $U_{2,1}\subseteq \AAA_{2,1}$ consisting of norm-1 linear maps is a fundamental domain (every $\psi\in \AAA_{2,1}$ is represented uniquely as $\varphi\pi$ for $\varphi\in \AAA_{1,1}$ and $\pi\in U_{2,1}$). It is clear that $U_{2,1}$ is simply $S^1$ (with $u\in S^1$ identified with $x\mapsto \langle x,u\rangle$). Fix a Riemannian metric $g$ on $U_{2,1}$. We can extend $g$ to $\AAA_{2,1}$ by translating by elements of $\AAA_{2,1}$: at (the tangent space at) $\varphi\pi\in \AAA_{1,1}U_{2,1}$ the inner product is the push-forward by $D\varphi$ of the inner product at $\pi$. It is easy to see that this is a smooth Riemannian structure and the associated metric is $\AAA_{1,1}$-invariant.

Alternatively, an invariant metric can defined explicitly: identify $\AAA_{2,1}$ with $\R\times S^1\times\R$  as above and define the inner product of vectors $u=(u_1,u_2,u_3)$ and $v=(v_1,v_2,v_3)$ in the tangent space at $(t,y,a)$ to be \[
g_{\psi}(u,v)=u_{1}v_{1}+\left\langle u_{2},v_{2}\right\rangle+e^{-2t}u_{3}v_{3},
\]
The reader may check invariance under the action.
\end{proof}

Throughout the paper, unless stated otherwise, we endow $\AAA_{2,1}$ with an $\AAA_{1,1}$-invariant metric $d$ as in the lemma. Balls and diameters are defined with respect to this metric.

\begin{lemma}\label{lem:diameter-bound}
Let $\rho,r,R>0$. Suppose that $E\subseteq \AAA_{2,1}$ and $F\subseteq\R^2$ are sets of diameter at most $r\leq 1$, that $F\subseteq B_R(0)$, and that $\rho=\|\psi\|$ for some $\psi\in E$. Then $EF=\{\psi(x)\,:\,\psi\in E\;,\;x\in F\}$ has diameter $O(\rho\cdot r \cdot R$).
\end{lemma}

\begin{proof}
  Let $E'=S_{1/\rho} T_{-\psi(0)}E$. By left invariance of the metric on $\AAA_{2,1}$ under left-multiplication by elements of $\AAA_{1,1}$, which $S_{1/\rho}$ and $T_{-\psi(0)}$ are, the diameter of $E'$ is at most $r$, hence at most $1$. Also, $E'$ intersects  the compact set of norm-one linear maps in $\AAA_{2,1}$, because $S_{1/\rho}T_{-\psi(0)}\psi\in E'$ is such a map. Thus, $E'$ lies in a fixed compact subset of $\AAA_{2,1}$, and it follows that $E'F$ has diameter $O(\diam(F))=O(rR)$. But \[
     EF=T_{\psi(0)}S_\rho(E'F),
     \]
     and the diameter of this set is  $\rho\diam(E'F)=O(\rho r R)$.
\end{proof}

In addition to $d$, it is also useful (and natural) to consider the  metric $D$ on $\AAA_{k,m}$ given by
\begin{equation}\label{eq:metricD}
D(\psi_1,\psi_2):=\max_{x\in B(0,1)}\|\psi_1(x)-\psi_2(x)\|,
\end{equation}
where $B(0,1)$ is the unit closed ball on the corresponding space $\R^k$. 

Recall that two metrics on the same space are bi-Lipschitz equivalent if with respect to the two metrics the identity map is bi-Lipschitz.

\begin{lemma}\label{lem:equivaence-of-metrics}
  \begin{enumerate}
    \item The metric $D$ on $\AAA_{2,1}$ is bi-Lipschitz equivalent to the $\AAA_{1,1}$-invariant metric $d$ on every compact subset of $\AAA_{2,1}$.
  \item The metric $D$ on $\AAA_{2,2}$ is bi-Lipschitz equivalent to the metric $d$ defined in the introduction as the pullback of the operator norm via the standard embedding $\AAA_{2,2}\hookrightarrow GL_3(\R)$. In particular, exponential separation can be defined equivalently using the metric $D$.
    \end{enumerate}
\end{lemma}

The proof is straightforward.

\subsection{$q$-adic partitions\label{sub:q-adic-partitions}}

Throughout the paper we will consider $q$-adic partitions of $\mathbb{R}$, where $q$ is a large integer defined in Section \ref{sub:dyadic-partitions} below (in fact all the main statements hold for $q=2$, but for some of the technical lemmas large $q$ is more convenient). The $q$-adic level-$n$ partition of $\R$ is defined by
\[
\DD_{n}=\left\{[\frac{k}{q^{n}},\frac{k+1}{q^{n}})\,:\,k\in\mathbb{Z}\right\}\:.
\]
We write $\DD_{t}=\DD_{[t]}$ when $t\in\mathbb{R}$
is non-integer. In $\mathbb{R}^{d}$ we write,
\[
\DD_{n}^{d}=\{I_{1}\times\ldots\times I_{d}\,:\,I_{i}\in\DD_{n}\},
\]
and generally omit the superscript.

We require similar partitions of $\AAA_{2,1}$. By \cite[Theorem~2.1]{KaenmakiRajalaSuomala2012}, there exists a collection of Borel sets
\[
\{Q_{n,i}\subset \AAA_{2,1}\::\:n\in\mathbb{Z},\:i\in \mathbb{N}\},
\]
having the following properties:
\begin{enumerate}
	\item $\AAA_{2,1}=\cup_{i\in \N}Q_{n,i}$ for every $n\in\mathbb{Z}$,
	\item $Q_{n,i}\cap Q_{m,j}=\emptyset$ or $Q_{n,i}\subset Q_{m,j}$ whenever
	$n,m\in\mathbb{Z}$, $n\ge m$, $i,j\in \N$,
	\item\label{item:ball} there exists a constant $C>0$ such that for every $n\in\mathbb{Z}$
	and $i\in\N$ there exists $\psi\in Q_{n,i}$ with
	\[
	B(\psi,C^{-1}q^{-n})\subset Q_{n,i}\subset B(\psi,Cq^{-n})\;.
	\]
\end{enumerate}
For each $n\in\mathbb{Z}$, denote $\DD_{n}^{\AAA_{2,1}}$
the partition $\{Q_{n,i}\::\:i\in\N\}$ of $\AAA_{2,1}$. With a slight abuse of notation, we usually omit the superscript.

\begin{lemma}\label{lem:Qn-has-bounded-degree}
  There exists a constant $C'$ such that for every $n\geq 0$ and $Q\in\QQ_n$,\[
    \#\{Q'\in\DD_{n+1}\::\:Q'\subset Q\}\leq C'\:.
  \]
\end{lemma}
\begin{proof}
Let $C$ be the constant as in property \eqref{item:ball} of the $q$-adic partitions, above. For any compact  $Z\subseteq \AAA_{2,1}$, there is some $C'=C'(Z)$ such that any  $Cr$-ball $B\subseteq Z$ contains at most $C'$ disjoint sub-$r/Cq$-balls. This is because on $Z$, the metric $d$ is bi-Lipschitz equivalent to $D$, and the statement clearly holds for $D$. Now, let $Q\in\QQ_n$ and let $Q_1,Q_2,\ldots\subset Q$ be disjoint elements with $Q_i\in \QQ_{n+1}$. Fixing some $\psi\in Q$, the set $\widehat{Q}=S_{1/\|\psi\|}T_{-\psi(0)}Q$ contains a norm-$1$ element and has diameter $O(q^{-n})=O(1)$, so it is contained in a fixed compact set $Z\subseteq \AAA_{2,1}$. By invariance of $d$, $\widehat{Q}$ is contained in a $Cq^{-n}$-ball, since this is true for $Q$. But the sets $\widehat{Q}_i$ defined similarly are disjoint subsets of $\widehat{Q}$, and each contains a $q^{-(n+1)}/C$-ball, since this is true for $Q_i$. Thus, the number of the $Q_i$ is bounded by $C'(Z)$, as desired.
\end{proof}

\subsection{\label{q-adic-components}$q$-adic components}

For a partition $\DD$ (in $\R$ or in $\AAA_{2,1}$ respectively) we write $\DD(x)$ for the unique partition element containing $x$. For a probability measure $\theta$, write
\[
\theta_{A}=\frac{1}{\theta(A)}\theta|_{A}
\]
for the conditional measure of $\theta$ on $A$, assuming $\theta(A)>0$. 

For a probability measure $\theta$ on a space equipped with partitions $\DD_{n}$, we define measure valued random variables $\theta_{x,n}$ such that $\theta_{x,n}=\theta_{\DD_{n}(x)}$ with probability $\theta(\DD_{n}(x))$. We call $\theta_{x,n}$ an $n$-th level component of $\theta$. When several components
appear, e.g. $\theta_{x,n}$ and $\tau_{y,n}$, we assume $x,y$ are
chosen independently. Sometimes $n$ is chosen randomly as well. For example, we write for $n_2\geq n_1$ integers and an event $\mathcal{U}$,
\begin{equation}\label{eq:p}
\PP_{n_1\leq i\leq n_2}(\mu_{x,i}\in\mathcal{U})=\frac{1}{n_2-n_1+1}\sum_{n=n_1}^{n_2}\PP(\mu_{x,n}\in\mathcal{U}).
\end{equation}
We write $\EE$ and $\EE_{n_1\leq i\leq n_2}$ for the expected value w.r.t. the probabilities $\PP$ and $\PP_{n_1\leq i\leq n_2}$.

\subsection{Partitions of symbolic space}\label{sub:dyadic-partitions}

The symbolic space $\Lambda^\N$ comes with the natural partitions into level-$n$ cylinder sets. It will be convenient to consider more general partitions into cylinders of varying length. Thus, if  $\Xi\subseteq\Lambda^*$ is a collection of words such that the cylinder sets corresponding to words in $\Xi$ form a partition of $\Lambda^\N$, then we say that $\Xi$ is a partition. In this case we also let $\Xi$ denote the associated ``name'' function $\Xi\colon\Lambda^{\N}\mapsto \Lambda^*$, so $\Xi(\ii)$ is the unique word in $\Xi$ such that $\ii\in[\Xi(\ii)]$.

Returning to our self-affine measure $\mu$, we first note that by iterating the basic identity $\mu=\sum_{\ii\in\Lambda}p_{\ii}\cdot\varphi_{\ii}\mu$,  for any partition $\Xi\subseteq\Lambda^*$, we get \[
\mu=\sum_{\ii\in\Xi} p_{\ii}\varphi_{\ii}\mu,
\]
and if $V\in\RP$ then by applying $\pi_V$ to the above, we get
\begin{equation}\label{eq:iterated-convolution-for-mu}
\pi_V\mu=\sum_{\ii\in\Xi}p_{\ii}\cdot\pi_{V}\varphi_{\ii}\mu.
\end{equation}

Let $q\geq2$ be an integer and let
$$
\Psi_n^q=\left\{(i_0,\ldots,i_m)\in\Lambda^*:\|A_{i_0,\ldots,i_m}\|\leq q^{-n}<\|A_{i_0,\ldots,i_{m-1}}\|\right\}.
$$
Thus, there exists a constant $c_0$, depending on the matrices but independent of $n$ such that for every $n\geq 1$ and for every $\ii\in\Psi_n^q$
$$
c_0 q^{-n} \leq \|A_{\ii}\|\leq q^{-n}.
$$
It is easy to see that $\Psi_{n}^q$ forms a partition of $\Lambda^{\N}$ for every $n\geq1$.

For a subspace $V\in\RP$, let $\Xi_n^{V,q}$ be the finite subset of $\Lambda^*$ such that $\|A_{\ii}^*|V\|\approx q^{-n+O(1)}$. That is,
$$
\Xi_n^{V,q}=\left\{(i_0,\ldots,i_m)\in\Lambda^*:\|A_{i_m}^*\cdots A_{i_0}^*|V\|\leq q^{-n}<\|A_{i_{m-1}}^*\cdots A_{i_{0}}^*|V\|\right\}.
$$
Thus, there exists a constant $c_0>0$ such that for every $\ii\in\Xi_n^{V,q}$ and every $V\in\RP$,
\begin{equation}\label{eq:c0}
c_0q^{-n}\leq\|A_{\ii}^*|V\|\leq q^{-n}.
\end{equation}

\begin{lemma}\label{lem:properties-of-q}
There exists an integer $q\geq2$ such that
\begin{enumerate}
  \item For every $k\neq j$, $\Psi_k^q\cap\Psi_j^q=\emptyset$,
  \item\label{item:disjointness-of-Xi-n} For every $k\neq j$ and $V\in\RP$, $\Xi_k^{V,q}\cap\Xi_j^{V,q}=\emptyset$,
  \item There exist $c_1=c_1(q)>0$ and $c_2=c_2(q)>0$ such that for every $n\geq1$ and $\ii\in\Psi_n^q$, $c_1 n\leq|\ii|\leq c_2n$,
  \item\label{item:that} There exist $c_1=c_1(q)>0$ and $c_2=c_2(q)>0$ such that for every $n\geq1$, $V\in\RP$ and $\ii\in\Xi_n^{V,q}$, $c_1n\leq|\ii|\leq c_2n$.
\end{enumerate}
\end{lemma}

\begin{proof}
We show only \eqref{item:disjointness-of-Xi-n} and \eqref{item:that}, the proof of the remaining parts are similar. To prove \eqref{item:disjointness-of-Xi-n}, it is enough to show that for every $j>k$, and for every $(i_0,\ldots,i_m)=\ii\in\Psi_k^q$, $\|A_{\ii}^*|V\|>q^{-j}$. But
\begin{align*}
\|A_{i_m}^*\cdots A_{i_0}^*|V\|&=\|A_{i_m}^*|A_{i_{m-1}}^*\cdots A_{i_0}^*V\|\cdot\|A_{i_{m-1}}^*\cdots A_{i_0}^*|V\|\\
&>q^{-k}\min_{i\in\Lambda}\|(A_i)^{-1}\|^{-1}>q^{-k-1}\geq q^{-j},
\end{align*}
for $q>1/\min_i\alpha_2(A_i)$. For \eqref{item:that}, observe that for $\ii\in\Xi_n^{V,q}$
$$
q^{-n}<\|A_{i_{m-1}}^*\cdots A_{i_0}^*|V\|\leq (\max_{i\in\Lambda}\alpha_1(A_i))^{|\ii|-1}.
$$
Hence, $|\ii|\leq n\frac{\log q}{-\log\max_i\alpha_1(A_i)}+1\leq2\frac{\log q}{-\log\max_{i}\alpha_1(A_i)}n$ holds for large enough $q$. The lower bound is similar with $|\ii|\geq \frac{\log q}{-\log\min_{i}\alpha_2(A_i)}n$.
\end{proof}

For the rest of the paper, we fix the integer $q$ satisfying the conclusion of Lemma~\ref{lem:properties-of-q}. All $q$-adic partitions are defined with respect to this $q$. We denote $\Psi_n^q$ by $\Psi_n$ and $\Xi_n^{V,q}$ by $\Xi_n^{V}$ for simplicity. In this spirit, when we write $\log$, we mean $\log_q$.

\subsection{\label{sub:random-cylinder-measures}Random cylinder measures}

Every measure on Euclidean space has associated to it its $q$-adic components. For a planar self-affine measure $\mu$, one can also decompose $\mu$ into cylinder measure, i.e. measure of the form $\varphi_\ii\mu$ for $\ii\in\Lambda^*$. As with $q$-adic components it is natural to view the cylinders as random measures, with the naturally defined weights.

For any given $n\in\mathbb{N}$ and $V\in\RP$, we introduce three random words $\II(n), \JJ(n,V)$ and $\UU(n)$ taking values from finite subsets (actually from partitions) of $\Lambda^*$.
\begin{itemize}
\item $\II(n)$ is the random word taking values in $\Psi_{n}$ according to the probability vector $p$, i.e.
  $$
  \PP(\II(n)=\ii)=\begin{cases}p_{\ii}&\text{ if }\ii\in\Psi_n,\\ 0&\text{ otherwise.}\end{cases}.
$$
\item $\JJ(n,V)$ is the random word taking values from $\Xi_n^V$ according to $p$.
\item $\UU(n)$ is the random word taking values from $\Lambda^n$ according to $p$.
\end{itemize}
If it is not confusing, sometimes we omit the second coordinate $V$ in $\JJ(n,V)$ and simply denote in by $\JJ(n)$.

We can also represent $\mu$ as a convex combination of cylinders. That is, equation \eqref{eq:iterated-convolution-for-mu} can be re-interpreted as
\begin{equation}\label{eq:decomposition-of-mu-into-levels}
\mu=\EE(\varphi_{\II(n)}\mu)=\EE(\varphi_{\JJ(n)}\mu)=\EE(\varphi_{\UU(n)}\mu).
\end{equation}

We may randomize $n$ as in the case of components, thus for example for any observable $F$,
\[
\mathbb{E}_{n_1\leq i\leq n_2}(F(\varphi_{\II(i)}\mu))=\frac{1}{n_2-n_1+1}\sum_{n=n_1}^{n_2}\mathbb{E}(F(\varphi_{\II(n)}\mu)).
\]

The random sequences $U(n),I(n)$ and $J(n,V)$ differ but statements which hold with high probability for one over many scales holds for the others. The following lemma describes the direction we need.

\begin{lemma}\label{lem:converting-Un-to-In-or-Jn}
  For every $V\in\RP$, let $\mathcal{U}_V\subseteq\Lambda^*$ be a set of words. Suppose that, for every $\varepsilon>0$ and $n\geq n(\varepsilon)$,  \[
  \inf_{V\in\RP}\PP_{1\leq i\leq n}(\UU(i)\in\mathcal{U}_V)>1-\varepsilon.
  \]
	Then for every $n\geq N'(\varepsilon)$
	\begin{eqnarray*}
&&\inf_{V\in\RP}\PP_{1\leq i\leq n}(\II(i)\in\mathcal{U}_V)>1-\varepsilon,\\
&&\inf_{V\in\RP}\PP_{1\leq i\leq n}(\JJ(i,V)\in\mathcal{U}_V)>1-\varepsilon.
        \end{eqnarray*}
        The same holds if we take the infimum over some fixed set of $V$'s in both the hypothesis and conclusion, or for a single $V$.
\end{lemma}

\begin{proof}
	By Lemma~\ref{lem:properties-of-q}, there exists a constant $c\geq1$ such that for every $\ii\in\Psi_n$, $|\ii|\leq n c$ and $\Psi_k\cap \Psi_j=\emptyset$ for $k\neq j$ as subsets of $\Lambda^*$.
	
	Let $\varepsilon>0$ be arbitrary. By assumption, there exists $N\geq1$ such that for every $n\geq N$ and every $V\in\RP$, $\PP_{1\leq i\leq n}(\UU(i)\in\mathcal{U}_V)>1-\varepsilon/2c$. 	Thus,
	\begin{align*}
	(c n-N)(1-\varepsilon/2c)&\leq cn\cdot \PP_{1\leq i\leq cn}(\UU(i)\in\mathcal{U}_V)\\
	&=n\cdot \PP_{1\leq i\leq n}(\II(i)\in\mathcal{U}_V)+cn\cdot \PP_{1\leq i\leq cn}(\UU(i)\in\mathcal{U}_V\setminus\bigcup_{\ell=1}^{n}\Psi_{\ell})\\
	&\leq n\cdot \PP_{1\leq i\leq n}(\II(i)\in\mathcal{U}_V)+cn-n.
	\end{align*}
	Thus, by choosing $\widehat{N}=\widehat{N}(\varepsilon,N(\varepsilon))$ such that $N/\widehat{N}<\varepsilon/2$, we get that for every $n\geq\widehat{N}$, $\inf_V\PP_{1\leq i\leq n}(\II(i)\in\mathcal{U}_V)>1-\varepsilon$. The proof for $\JJ(i,V)$ is similar.
\end{proof}

\subsection{Entropy}\label{sub:entropy}

Denote $H(\mu,\DD)$ the usual entropy w.r.t the partition $\DD$, and denote $H(\mu,\DD'|\DD)$ the usual conditional entropy. That is,
\begin{eqnarray}
  H(\mu,\DD) & = & -\int\log\mu(\DD(x))d\mu(x)\\
  H(\mu,\DD'|\DD) &= & H(\mu,\DD'\vee\DD)-H(\mu,\DD),\label{eq:conditional-entropy-is-increment}\\
  & = & \sum_{I\in \DD}\mu(I)\cdot H(\mu_I,\DD'),\label{eq:cond-entropy-as-expectation}
\end{eqnarray}
where $\DD'\vee\DD$ denotes the common refinement of the partitions $\DD',\DD$. By the definition of the distribution on components,
\begin{eqnarray}
H(\mu,\DD_{n+m}|\DD_n)  & = &   \EE(H(\mu_{x,n},\DD_{n+m})). \label{eq:component-entropy-is-conditional-entropy}
\end{eqnarray}

The entropy functions is concave and almost convex in the measure argument. That is, for any $0\leq\alpha\leq1$ and $\mu_1,\mu_2$ probability measures
\begin{eqnarray*}
	\alpha H(\mu_1,\DD)+(1-\alpha)H(\mu_2,\DD)&\leq&H(\alpha\mu_1+(1-\alpha)\mu_2,\DD)\\
	&\leq&\alpha H(\mu_1,\DD)+(1-\alpha)H(\mu_2,\DD)+H(\alpha),
\end{eqnarray*}
where $H(\alpha)=-\alpha\log\alpha-(1-\alpha)\log(1-\alpha)$.

Scale-$n$ entropy transforms nicely under affine maps: For any $f\in \AAA_{1,1}$
\begin{eqnarray}
H(f\mu,\DD_n)&=&H(\mu,\DD_{n+\log\|f\|})+O(1)\text{ and}\label{eq:entropy-under-transformation-1}\\
H(f\mu,\DD_n)&=&H(\mu,\DD_n)+O(\log\|f\|+1).\label{eq:entropy-under-transformation-2}
\end{eqnarray}
Moreover, the entropy of images is continuous in the map: If $f,g\colon\R\mapsto\R$ are such that $\sup_x|f(x)-g(x)|<q^{-n}$ then
\begin{equation}\label{eq:entropy-under-transformation-3}
|H(f\mu,\DD_n)-H(g\mu,\DD_n)|=O(1).
\end{equation}

The entropy function $\mu\mapsto H(\mu,\QQ_n)$ is continuous in the total variation norm $\|\cdot\|$. In fact, if $\|\mu-\nu\|<\ep$ and $\mu,\nu$ are supported on $k$ atoms of $\DD'$, then (\cite[Lemma 3.4]{Hochman2015}):
\begin{equation}
|H(\mu,\DD'|\DD)-H(\nu,\DD'|\DD)| < 2\log k\ep + 2H(\frac{\ep}{2}).\label{eq:entropy-continuity-of-images}
\end{equation}
In particular, using \ref{eq:cond-entropy-as-expectation} and the fact that each $I\in\DD_m$ intersects $2^{n-m}$ atoms of $\DD_n$, this implies
\begin{equation}
  |\frac{1}{n-m}H(\mu,\QQ_n|\QQ_m)-\frac{1}{n-m}H(\nu,\QQ_n|\QQ_m)| < 2\ep+2H(\ep). \label{eq:scale-n-entropy-continuity-in-TV}
\end{equation}
Entropy  is not continuous in the weak-* topology. Nevertheless it is easy to introduce a continuous substitute. For example, one can replace the integrand $\log\mu(\QQ_n(x))$, which is a step function, by a continuous approximation, using a partition of unity to approximate the indicators of the level sets. Alternatively one can average the entropy over translations, as in \cite{Varju2016}. These ``alternative'' entropy functions can be made so that the difference from the scale-$n$ entropy is of order $O(1)$. Since we shall nearly always deal with the asymptotics of normalized entropies such as $(1/n)H(\mu,\QQ_n)$, such a change is insignificant, and we will freely assume, when the need arises, that we are using such a substitute. 

\subsection{Entropy dimension\label{sub:entropy-dimension}}
For a $\mu\in\mathcal{P}(\R)$, let $\underline{\dim}_e\mu$ be the lower- and let $\overline{\dim}_e\mu$ be the upper-entropy dimension. That is,
\begin{eqnarray*}
  \underline{\dim}_e\mu & = & \liminf_{n\to\infty}\frac{H(\mu,\DD_n)}{n}\\
  \overline{\dim}_e\mu &= &\limsup_{n\to\infty}\frac{H(\mu,\DD_n)}{n}.
\end{eqnarray*}
If the limit exists then we call it the entropy dimension of $\mu$ and we denote it by $\edim\mu$.

\begin{lemma}\label{lem:entropy-dimension}
If $\mu\in\mathcal{P}(\R)$ is exact dimensional then $\dim_e\mu$ exists, moreover,
$$
\dim\mu=\lim_{n\to\infty}\frac{H(\mu,\DD_n)}{n}.
$$
\end{lemma}

The proof of the lemma can be found in \cite[Theorem~4.4]{Young} or \cite[Theorem~1.3]{FanLauRao}. We refer the reader to \cite[Section 3.1]{Hochman2015} and \cite[Section~2.3 and 2.4]{hochman2016some} for further properties of entropy.

\subsection{Furstenberg measure}\label{sub:furstenberg}
For $A\in GL_2(\R)$, let $\overline{A}=|\det(A)|^{-1/2}A$. Let $\tau^+=\sum_{i\in\Lambda}p_i\delta_{\overline{A}_i}$, $\tau^-=\sum_{i\in\Lambda}p_i\delta_{\overline{A}_i^{-1}}$ and $\tau^*=\sum_{i\in\Lambda}p_i\delta_{\overline{A}_i^*}$ be measures on $GL_2(\R)$. We let $GL_2(\R)$ act on $\RP$ in the natural way (linear maps take lines to lines), and for a measure $\theta$ on $GL_2(\R)$ and a measure $\tau$ on $\RP$ we write $\theta\conv \tau$ for the push-forward of $\theta\times\tau$ by the map $(A,x)\mapsto Ax$.

Denote $\sigma$ the left-shift operator on $\Lambda^\N$ and $\nu=p^{\N}$ the product measure. We introduce three functions on the space $\Lambda^{\N}\times\RP$ as follows
\begin{eqnarray*}
&&P_+(\ii,V)=(\sigma\ii,A_{i_0}V),\\
&&P_*(\ii,V)=(\sigma\ii,A_{i_0}^*V),\\
&&P_-(\ii,V)=(\sigma\ii,A_{i_0}^{-1}V).
\end{eqnarray*}

\begin{proposition}\label{prop:furstenberg-measure}
  If $\{\overline{A}_i\}_{i\in\Lambda}$ generates a strongly irreducible and unbounded subgroup of $GL_2(\R)$ then  for every choice $a=*,+,-$, there exists a unique probability measure $\eta^a$ on $\RP$ such that $\tau^a.\eta^a=\eta^a$ and the measure $\m\times\eta^a$ is $P_a$-invariant and mixing. Moreover, there exist $0>\chi_1>\chi_2$ satisfying
  \begin{align*}
  &\lim_{n\to\infty}\frac{1}{n}\log\|A_{i_n}^{-1}\cdots A_{i_0}^{-1}|V\|=-\chi_2,\text{ for $\m\times\eta^-$-a.e. }(\ii,V),\\
  &\lim_{n\to\infty}\frac{1}{n}\log\|A_{i_n}\cdots A_{i_0}|V\|=\chi_1,\text{ for $\m\times\eta^+$-a.e. }(\ii,V)\text{, and}\\
  &\lim_{n\to\infty}\frac{1}{n}\log\|A_{i_n}^*\cdots A_{i_0}^*|V\|=\chi_1,\text{ for $\m\times\eta^*$-a.e. }(\ii,V).
  \end{align*}
\end{proposition}

For the proof of the proposition, we refer to \cite[Chapter III]{BougerolLacroix1985} and \cite[Theorem~3.4.1]{Arnold}. We call $\eta^+$ the forward Furstenberg measure and $\eta^-$ the backward Furstenberg measure.

It is not hard to show that $R\eta^-=\eta^*$, where $R(V)=V^{\perp}$.

\begin{proposition}\label{prop:convergence-to-furstenberg-measure}\label{lem:uniform-convergence-to-furst-measure}
 Assume that $\{\overline{A}_i\}_{i\in\Lambda}$ generates a strongly irreducible and unbounded subgroup of $GL_2(\R)$. Then the distribution of the random line $A^*_{\UU(n)}V$ converges to $\eta^*$, and for every $V\in\RP$
  $$ 
  \lim_{n\to\infty}\frac{1}{n}\sum_{k=0}^{n-1}\delta_{A_{i_k}^*\cdots A_{i_0}^*V}=\eta^*\text{ for $\m$-a.e. $\ii=(i_0,i_1,\ldots)$}.
  $$
Furthermore, the convergence is uniform in $V$, in the sense that if $f\colon\RP\mapsto\R$ is continuous, then
\[
\lim_{n\to\infty}\sup_{V\in\RP}\left|\EE(f(A_{\UU(n)}^*V))-\int fd\eta^*\right|=0.
\]%
In particular, for every $E\subseteq\RP$ open set, for every $\varepsilon>0$, for every $n\geq N(E,\varepsilon)\geq1$ and for every $V\in\RP$
$$
\PP\left(A_{\UU(n)}^*V\in E\right)\geq\eta^*(E)-\varepsilon.
$$
\end{proposition}

The statement follows from \cite[Chapter III]{BougerolLacroix1985}.

\subsection{Dimension of projections}

The function $(\ii,V)\mapsto \dim \pi_V\mu$, as a function on $\Lambda^\N\times\RP$, is monotone under the skew-product dynamics of the previous section (it increases or decreases depending on whether we consider the lower or upper Hausdorff dimension of the measure), so by ergodicity, its value is a.s. constant. We record a slightly stronger (and deeper) conclusion in the next lemma.
 
\begin{lemma}\label{lem:alpha}
   Assume that $\{\overline{A}_i\}_{i\in\Lambda}$ generates a strongly irreducible and unbounded subgroup of $GL_2(\R)$. Then there exists $0\leq\alpha\leq1$ such that $\eta^*$-a.e. projection $\pi_V\mu$ is exact dimensional with dimension $\alpha$, and
  $$
  \lim_{n\to\infty}\frac{1}{n}H(\pi_V\mu,\DD_n)=\alpha\;;\text{ for $\eta^*$-a.e. $V$}.
  $$
\end{lemma}
\begin{proof}
  By \cite[Theorem 2.3]{BaranyKaenmaki2015}, there exists $\alpha$ such that $\pi_{V^{\perp}}\mu$ is exact dimensional and $\dim_H\pi_{V^{\perp}}\mu=\alpha\text{ for $\eta^-$-a.e. }V\in\RP$. Since $R\eta^-=\eta^*$, where $R(V)=V^{\perp}$, this implies the statement with Lemma~\ref{lem:entropy-dimension}.
\end{proof}

\section{\label{sec:entropy-porosity}Entropy porosity of orthogonal projections}

In order to analyze entropy growth under convolution it will be necessary to show that the projections of  $\mu$ are entropy porous in the sense of \cite{hochman2016some}:

\begin{definition}\label{def:porosity}
  Let $\rho\in\mathcal{P}(\R)$. We say that $\rho$ is {\it $(h,\delta,m)$-entropy porous from scale $n_1$ to $n_2$} if
  $$
  \PP_{n_1\leq i\leq n_2}\left(\frac{1}{m}H(\rho_{x,i},\DD_{i+m})<h+\delta\right)>1-\delta.
  $$
  
\end{definition}
 
Let $\alpha$ be the $\eta^*$-typical (entropy) dimension of projections of $\mu$. The main purpose of this section is to prove:

\begin{proposition}\label{prop:projections-are-porous}
  Assume that $\{\overline{A}_i\}_{i\in\Lambda}$ generates a strongly irreducible and unbounded subgroup of $GL_2(\R)$. For every $\varepsilon>0$, $m\geq M(\varepsilon)$, $k\geq K(\varepsilon,m)$ and $n\geq N(\varepsilon,m,k)$, for every $\psi\in \AAA_{2,1}$ and writing $t=\log\|\psi\|$,
\begin{equation}\label{eq:projections-are-porous}
 \PP_{1\leq i\leq n}\left(\substack{\text{\large $\psi\varphi_{\II(i)}\mu$ is $(\alpha,\varepsilon,m)$-entropy porous}\\ \text{\large from scale $t+i$ to $t+i+k$}}\right)>1-\varepsilon.
\end{equation}

\end{proposition}

The proof strategy is similar to the proof of porosity for self-similar measures. We first show that at most scales, the components of $\pi_V\mu$ have entropy at least $\alpha$, by covering them with small projections of cylinder measures. But the average of the entropies of components is essentially the entropy dimension of $\pi_V\mu$, which is again $\alpha$. Thus, since the average is not much larger than the pointwise lower bound on component entropy, we get a corresponding pointwise upper bound, which is entropy porosity.

Arguments of this type have appeared elsewhere but under slightly stronger assumptions. We give a self-contained proof, and in the hope of preventing further repetitions we state some general lemmas below which perhaps can be re-used.

\subsection{\label{sub:covering-lemma-and-porosity}Covering arguments and porosity}

In this section we show quite generally that if a measure $\tau$ decomposes as a convex combination of measures $\tau_i$, supported on short intervals, then many properties of the $\tau_i$, and specifically their entropies, are inherited by the $q$-adic components of $\tau$. We then derive sufficient conditions for entropy porosity and invariance of entropy porosity under affine coordinate changes. 

\begin{lemma}\label{lem:total-variation-bound}
For every $\varepsilon>0$ there exists a $\delta>0$ with the following
property. Suppose that a probability measure $\tau\in\mathcal{P}(\mathbb{R})$
can be written as a convex combination $\tau=(1-\delta)\tau'+\delta\tau''$.
Then for every $k$,
\[
\tau\left(x\,:\,\left\Vert \tau_{x,k}-\tau'_{x,k}\right\Vert <\varepsilon\right)>1-\varepsilon.
\]
(in fact we can take any $\delta<\varepsilon^{2}/4$).\end{lemma}
\begin{proof}
  Let $\varepsilon$ be given, and fix $\delta<1/2$. For any set $I$ of positive $\tau$-measure, algebraic manipulation shows that \[
  \tau_I = (1-\delta)\frac{\tau'(I)}{\tau(I)}\tau'_I + \delta\frac{\tau''(I)}{\tau(I)}\tau''_I.
  \]
Therefore,\[
  \|\tau_I-\tau'_I\| = \delta\frac{\tau''(I)}{\tau(I)} \|\tau'_I - \tau''_I\| \leq 2\delta\frac{\tau''(I)}{\tau(I)}. 
\]
It remains to show that if $\delta$ is small enough then $\delta\tau''(I)/\tau(I)<\varepsilon/2$ for $\tau$-most intervals $I\in\mathcal{Q}_k$. To this end, consider the measures $\widehat{\tau},\widehat{\tau}'$ and $\widehat\tau''$ induced on the countable probability space $\mathcal{Q}_k$ by $\tau,\tau',\tau''$ on $\R$: that is, $\widehat{\tau}(\{I\})=\tau(I)$ and similarly for $\tau',\tau''$. Clearly $\widehat\tau = (1-\delta)\widehat\tau'+\delta\widehat\tau''$, so $\widehat\tau''\ll\widehat\tau$. Let $f=d\widehat\tau''/d\widehat\tau$, so $f(I)=\tau''(I)/\tau(I)$. Then $\int fd\widehat\tau=1$ and $f\geq 0$, so by Markov's inequality, \[
     \widehat\tau\{I\,:\,\delta f(I)>\sqrt\delta\}<\sqrt\delta.
  \]
This proves the claim for any $\delta<\min\{\varepsilon^2/4\}$.
\end{proof}

\begin{corollary}\label{cor:covering-by-high-entropy-measures}\label{cor:covering-by-low-entropy-measures}
For every $\varepsilon>0$ there exists a $\delta>0$ with the following
property. Suppose that a probability measure $\tau\in\mathcal{P}(\mathbb{R})$
can be written as a convex combination $\tau=(1-\delta)\tau'+\delta\tau''$,
and that for some $\alpha>0$ and  $m,k\in\N$ we have
\[
\frac{1}{m}H(\tau'_{x,k},\mathcal{Q}_{k+m})\geq\alpha\qquad\mbox{for }\tau'\mbox{-a.e. }x.
\]
Then
\[
\mathbb{P}_{i=k}\left(\frac{1}{m}H(\tau_{x,k},\mathcal{Q}_{k+m})\geq\alpha-\varepsilon\right)>1-\varepsilon.
\]
Also, if for some $\beta$ we have 
\[
\frac{1}{m}H(\tau'_{x,k},\mathcal{Q}_{k+m})\leq\beta\qquad\mbox{for }\tau'\mbox{-a.e. }x.
\]
then
\[
\mathbb{P}_{i=k}\left(\frac{1}{m}H(\tau_{x,k},\mathcal{Q}_{k+m})\leq\beta+\varepsilon\right)>1-\varepsilon.
\]
\end{corollary}
\begin{proof}
  By \eqref{eq:scale-n-entropy-continuity-in-TV}, there is a $\rho>0$ such that if a pair of components satisfy $\|\tau'_{x,k}-\tau_{x,k}\|<\rho$, then
  \[
  |\frac{1}{m}H(\tau'_{x,k},\QQ_{k+m}) - \frac{1}{m}H(\tau_{x,k},\QQ_{k+m})| < \varepsilon.
\]
Thus if $\delta$ is small enough compared to this $\rho$, the lemma follows from the previous lemma.
\end{proof}

The following lemma says that if a measure $\tau$ decomposes into
measures on short intervals, each of  which has large entropy, then the
components of $\tau$, at the same scales, have large entropy. Thus, large entropy transfers from an ``arbitrary'' decomposition to the component decomposition. In the statement of the proposition we fix a scale $k$, a shorter scale $k+\ell$ for the intervals supporting the  measures $\tau_i$ which make up $\tau$, and an even shorter scale $k+m$ at which the entropy  appears. The dependence between these parameters is that $\ell$ is large but fixed, $m\gg \ell$, and $k$ is arbitrary. We in fact do not require that $m\gg\ell$ explicitly, but if this fails then the entropy cannot be as large as required in the lemma.
\begin{lemma}\label{lem:entropy-porosity-at-one-level-from-covering}
For every $\varepsilon>0$ there exists a $\delta>0$ with the following
property. Let $\tau\in\mathcal{P}(\mathbb{R})$ be written as a convex
combination $\tau=(\sum_{i=1}^N p_{i}\tau_{i}) + p_0\tau_0$, with $p_0<\delta$, and suppose that for some
$m,k,\ell\in\N$ and $\alpha>0$,
\begin{enumerate}
\item \label{enu:high-entropy}$\frac{1}{m}H(\tau_{i},\mathcal{Q}_{k+m})>\alpha$ for every $i=1,\ldots,N$.
\item \label{enu:short-support}$\tau_{i}$ is supported on an interval of length $\leq q^{-(k+\ell)}$ for every $i=1,\ldots,N$.
\item \label{enu:small-boundaries} $\tau(I)<\delta\tau(J)$ whenever $I\subseteq J$ are concentric intervals, $|J|=q^{-k}$ and $|I|=q^{-\ell}|J|$.
\end{enumerate}
Then
\begin{equation}
\mathbb{P}_{i=k}\left(\frac{1}{m}H(\tau_{x,i},\mathcal{Q}_{i+m})>\alpha-\varepsilon\right)>1-\varepsilon . \label{eq:large-component-entropy-at-single-level}
\end{equation}
\end{lemma}
\begin{remark}
In (\ref{enu:high-entropy}) we could have assumed instead that $\frac{1}{m}H(\tau_{i},\mathcal{Q}_{k+\ell+m})>\alpha$
for every $i$, and that $m>m(\ell)$, since this condition implies
the condition (\ref{enu:high-entropy}) as stated. This may be more natural in applications
since $\tau_{i}$ is supported on an interval of order $q^{-(k+\ell)}$,
so $\mathcal{Q}_{k+\ell+m}$ is ``$m$ scales smaller''. \end{remark}

Also, in (\ref{enu:small-boundaries}), we actually only care about the case that $I$ is a ball centered at a rational $n/q^k$, since this implies that $\tau$ gives small mass to the $q^{-(k+\ell)}$-neighborhood of the set of endpoints of intervals in $\mathcal{Q}_k$. But the condition above is more natural.

\begin{proof}
Let $A'\subseteq\{1,\ldots,N\}$ denote the set of indices $1\leq i\leq N$ such that $\tau_{i}$ is
supported on a $\mathcal{Q}_{k}$-cell, and $A''$ the remaining indices, including $0$.
Let $a'=\sum_{i\in A'}p_{i}$ and $a''=\sum_{i\in A''}p_{i}$, let
$p'_{i}=p_{i}/a'$ and $p''_{i}=p_{i}/a''_{i},$ and finally, $\tau'=\sum_{i\in A'}p'_{i}\tau_{i}$
and $\tau''=\sum_{i\in A''}p''_{i}\tau_{i}$. We thus have $\tau=a'\tau'+a''\tau''$.
Also, every level-$k$ component of $\tau'$ is a convex combination
of $\tau_{i}$'s, so by concavity of entropy and hypothesis (\ref{enu:high-entropy}), $\frac{1}{m}H(\tau'_{x,k},\mathcal{Q}_{k+m})>\alpha$
for $\tau'$-a.e. $x$. By Corollary \ref{cor:covering-by-high-entropy-measures}, we will be done
if we show that $a'$ is arbitrarily close to one when $\delta$ is
small enough. To see this, note that if $i\in A''\setminus\{0\}$, then by hypothesis
(\ref{enu:short-support}), $\tau_{i}$ is supported in a $q^{-(k+\ell)}$-ball centered
at the endpoint of some $\mathcal{Q}_{k}$-cell. But by hypothesis
(\ref{enu:small-boundaries}), the total $\tau$-mass of these balls is at most $2\delta$,
and this mass is at least $a''-p_0$. Using $p_0<\delta$ we get that $1-a'=a''<3\delta$. This proves
the claim.\end{proof}
\begin{lemma}\label{lem:entropy-porosity-from-covering}
For every $\varepsilon>0$ there exists a $\delta>0$ with the following
property. Let $m,\ell\in\N$ and $n>n(m,\ell)$ be given, and suppose
that $\tau\in\mathcal{P}(\mathbb{R})$ is a measure such that for
a $(1-\delta)$-fraction of $1\leq k\leq n$, we can write $\tau=\sum p_{i}\tau_{i}$
so as to satisfy the three conditions of the previous lemma, for the
given $\delta$ and $\ell,m,k$. Assume further that $|\frac{1}{n}H(\tau,\mathcal{Q}_{n})-\alpha|<\delta$.
Then $\tau$ is $(\alpha,\varepsilon,m)$-entropy porous from scale
$1$ to $n$.\end{lemma}
\begin{remark}
  We could again assume the weaker version of (\ref{enu:high-entropy}) in Lemma \ref{lem:entropy-porosity-at-one-level-from-covering}, see remark after that lemma.
\end{remark}

\begin{proof}
Fix $\delta$, and assume it is small enough to satisfy the previous
lemma, and assume also that $\delta<\varepsilon$. Then whenever $1\leq k\leq n$
allows a representation as in the previous lemma, we know that (\ref{eq:large-component-entropy-at-single-level})
holds. By hypothesis such a representation exists for a $(1-\delta)$-fraction
of of $1\leq k\leq n$, hence 
\[
\mathbb{P}_{1\leq i\leq n}\left(\frac{1}{m}H(\tau_{x,i},\mathcal{Q}_{i+m})\geq\alpha-\varepsilon\right)>(1-\delta)(1-\varepsilon)>1-2\varepsilon,
\]
where we used $\delta<\varepsilon$. But 
\begin{eqnarray*}
  \frac{1}{n}H(\tau,\mathcal{Q}_{n}) & = & \frac{1}{n}H(\tau,\DD_1) + \mathbb{E}_{1\leq i\leq n}\left(\frac{1}{m}H(\tau_{x,i},\mathcal{Q}_{i+m})\right)+O(\frac{m}{n})\\
  & \geq & \mathbb{E}_{1\leq i\leq n}\left(\frac{1}{m}H(\tau_{x,i},\mathcal{Q}_{i+m})\right)+O(\frac{m}{n})
\end{eqnarray*}
so, since we assume $n$ large relative to $m$, the error may be
made less than $\varepsilon$, so by hypothesis, 
\[
\mathbb{E}_{1\leq i\leq n}\left(\frac{1}{m}H(\tau_{x,i},\mathcal{Q}_{i+m})\right)<\frac{1}{n}H(\tau,\mathcal{Q}_{n})+\varepsilon<\alpha+2\varepsilon.
\]
Thus the integrand in the expectation above is bounded below by the
mean, up to a $4\varepsilon$ error. A corresponding upper bound follows,
showing that with probability $1-\Theta(\sqrt{\varepsilon})$ component
entropy is bounded above by the mean up to a $\Theta(\sqrt{\varepsilon})$
error. This is what we wanted if we begin with a smaller $\varepsilon$.
\end{proof}
A variant of the Lemma \ref{lem:entropy-upper-bound-at-one-level-from-covering} gives upper bounds on component entropy, as long as there are not too many measures $\tau_i$ in each $q^{-k}$-interval.
\begin{lemma}\label{lem:entropy-upper-bound-at-one-level-from-covering}
For every $\varepsilon>0$ there exists a $\delta>0$ with the following
property. Let $\tau\in\mathcal{P}(\mathbb{R})$ be written as a convex
combination $\tau=\sum p_{i}\tau_{i}$, and suppose that for some
$k,\ell,p\in\N$,  $m>m(\ep,p)$ and $\beta>0$, 
\begin{enumerate}
\item $\frac{1}{m}H(\tau_{i},\mathcal{Q}_{k+m})<\beta$ for every $i$.
\item Every $\tau_{i}$ is supported on an interval of length $\leq q^{-(k+\ell)}$.
\item  $\tau(I)<\delta\tau(J)$ whenever $I\subseteq J$ are concentric intervals, $|J|=q^{-k}$ and $|I|=q^{-\ell}|J|$.
\item Every interval of length $q^{-k}$ intersects the support of at most
$p$ of the measures $\tau_{i}$.
\end{enumerate}
Then
\begin{equation}
\mathbb{P}_{i=k}\left(\frac{1}{m}H(\tau_{x,i},\mathcal{Q}_{i+m})<\beta+\varepsilon\right)>1-\varepsilon .\label{eq:low-component-entropy-at-single-level-1}
\end{equation}
\end{lemma}
\begin{proof}
The proof is identical to the proof of Lemma \ref{lem:entropy-porosity-at-one-level-from-covering}, the only
difference being that instead of concavity of entropy, we use almost
convexity (see Section \ref{sub:entropy}). This introduces an error term which, by hypothesis (4),
is of order $O((\log p)/m)$. Since $m$ is assumed large relative
to $p$, this can be made negligible. \end{proof}
\begin{lemma}\label{lem:porosity-invarance-under-affine-maps}
For every $\varepsilon>0$ there exists a $\delta>0$ with the following
property. Let $\ell\in\N$ and $m>m(\ep,\ell)$, and let $\tau\in\mathcal{P}(\mathbb{R})$ be a measure such that hypothesis (\ref{enu:small-boundaries}) of the previous lemma
holds for every $k$. Let $n>n(m,\ell)$, and suppose that $\tau$ is $(\alpha,\delta,m)$-entropy
porous from scale $n_{1}$ to $n_{2}=n_{1}+n$. Then for any affine
map $f(x)=ax+b$, the measure $f\tau$ is $(\alpha,\varepsilon,m)$-entropy
porous from scale $n_{1}-\log |a|$ to $n_{2}-\log |a|$.\end{lemma}
\begin{proof}
Apply the previous lemma: at a large fraction of scales $q^{-i}$ taking $\tau_{i}$
to be the $f$-images of level-$(i+\ell)$ components of $\tau$; we also need to discard a small fraction of the components at each scale to satisfy assumption (1) of the lemma. The parameter $p$ in (4) of the lemma is $p=q^\ell$. \end{proof}

\subsection{\label{sub:uniform-entropy-across-scales}Uniform continuity across scales} 

To verify the continuity condition (\ref{enu:small-boundaries}) in Lemmas \ref{lem:entropy-porosity-at-one-level-from-covering} and \ref{lem:entropy-upper-bound-at-one-level-from-covering}, we introduce the following stronger notion.

\begin{definition}\label{def:uniform-entropy-across-scales}
	We say that a measure $\tau\in\mathcal{P}(\R)$ is uniformly continuous across scales if for every $\varepsilon>0$ there exists $\delta=\delta(\varepsilon)>0$ such that for any $x\in\R$ and $r>0$
	\[
	\tau(B(x,\delta r))\leq\varepsilon\cdot\tau(B(x,r)).
	\]
        A family of measures is jointly uniformly continuous across scales if all measures in the family satisfy this condition with a common function $\delta(\cdot)$.
\end{definition}

We return now to our self-affine measure $\mu$. We aim to prove uniform continuity across scales of its 1-dimensional projections. We first prove:

\begin{lemma}\label{lem:uniform-continuity-across-scales}
	Assume that $\{\overline{A}_i\}$ act irreducibly on $\R^2$. Then for every $\delta>0$ there exists $\rho>0$ such that for every $V\in\RP$ and $x\in\R$, $\pi_V\mu(B(x,\rho))<\delta$.
\end{lemma}

\begin{proof}
  The family $\{\pi_V\mu\}_{V\in\RP}$ is compact in the weak-* topology, and it is not hard to see that if the statement of the lemma fails, then this would imply that some $\pi_V\mu$ has an atom, i.e., that there exists $V\in\RP$ and $x\in\R$ with $\pi_V\mu(\{x\})>0$. This is the same as $\mu(\pi_V^{-1}(x))>0$, so $\mu$ gives positive mass to an affine line. But it is easy to see this contradicts the irreducibility assumption.
\end{proof}

\begin{lemma}\label{lem:joint-uniform-continuity}
	Assume that $\{\overline{A}_i\}$ act irreducibly on $\R^2$. Then the family of measures $\pi_V\mu$ ($V\in\RP$) is jointly uniformly continuous across scales.
\end{lemma}

\begin{comment}
For every $\varepsilon>0$ and $r\geq1$ there exists $\ell=\ell(\varepsilon,r)>0$ such that for every $f\in W_r^{1,1}$ and $n\geq1$
$$
\sup_{V\in\RP}\pi_V\mu\left(\bigcup_{k=-\infty}^{\infty}B(f^{-1}(\frac{k}{q^{-n}}),q^{-(n+\ell)})\right)<\varepsilon.
$$
\end{comment}

\begin{proof}
  For simplicity assume that $\mu$ is supported on the unit ball, the general case can be derived from this or proved similarly. Let $0<c<1$ be such that $cq^{-n}\leq \|\varphi_{\JJ(n,V)} \|\leq q^{-n}$.  Let $\varepsilon>0$ be arbitrary. Then by Lemma~\ref{lem:uniform-continuity-across-scales}, there exists $\delta>0$ such that $\sup_x\sup_V\pi_V\mu(B(x,\delta/c))<\varepsilon$.

Fix a ball $B(x,r)$ and $V\in\RP$. Then
\begin{align*}
  \pi_V\mu(B(x,\frac{\delta r}{3}))&=\EE\left(\pi_V\varphi_{\JJ(-\log_q(r/3),V)}\mu(B(x,\frac{\delta r}{3}))\right).
\end{align*}
The expression in the expectation is of the form $T_aS_t\pi_W\mu(B(x,\delta r/3))$ for some $a\in \R$, $t>0$ and $W\in\RP$. By definiiton of $\JJ(,)$, the scaling $t$ is in the range $(cr/3,r/3]$. Re-scaling this expression by $1/t$, it is the mass of a ball of radius $<\delta/c$ with respect to $\pi_W\mu$, which, by choice of $\delta$, is less than $\varepsilon$. But this is the  contribution only if the mass is positive. Conditioning on this event we have (using the assumption that $\mu$ is supported on $B(1,0)$):
\begin{align*}
&\leq\varepsilon\PP(\pi_V\varphi_{\JJ(-\log_q(r/3),V)}(B(0,1))\cap B(x,\frac{\delta r}{3})\neq\emptyset)\\
&\leq\varepsilon\pi_V\mu(B(x,r)).\qedhere
\end{align*}
\end{proof}

\subsection{\label{sub:lower-bounds-for-projected-component-entropy}\label{sub:porosity}Entropy porosity for $\pi_V\mu$}

Our eventual goal is to apply Lemma \ref{lem:entropy-porosity-from-covering} to  $\pi_V\mu$ and establish its entropy porosity. In order to verify assumption \eqref{enu:high-entropy} of that lemma, we prove in this section that most components of $\pi_V\mu$ have entropy close to $\alpha$, where $\alpha$ denotes the $\eta^*$-a.s. value of $\dim \pi_V\mu$:

We begin with an analysis of  $\|A^*_{\UU(i)} | V \|$.
For every matrix $A$ and subspace $V$, we always have $\|A\|\geq \|A|V\|$. But if we fix $A$ and let $V$ vary, then as long as $V$ stays away from $\ker A$, which it typically does, we will have $\|A\|\leq C \|A|V\|$, where $C$ depends on the distance of $V$ from $\ker A$. Since $\eta^*$ is a non-atomic measure, only a vanishing amount of its mass comes close to any fixed subspace. This is essentially the proof of

\begin{lemma}\label{lem:bounding-norms-on-subspaces}
	For every $\varepsilon>0$ there exist $C=C(\varepsilon)>0$ and $N=N(\varepsilon)\geq1$ such that for every $V\in\RP$ and $n\geq N$,
	\begin{eqnarray*}
		\PP\left(\|A^*_{\II(n)} | V\|\leq \|A^*_{\II(n)}\|  < C\|A_{\II(n)}^*|V\|\right) & > & 1-\varepsilon,\\
                \PP\left(\|A^*_{\UU(n)} | V\|\leq\|A^*_{\UU(n)}\|  <  C\|A_{\UU(n)}^*|V\|\right) & > & 1-\varepsilon.
	\end{eqnarray*}
%(The same is true without the transposes for a set of $n$ of density $1$).
\end{lemma}

The formal proof relies on a combination of \cite[Proposition~III.3.2]{BougerolLacroix1985} and Egorov's Theorem, we omit the details.

Next, we show that projections of  typical cylinders have high entropy at smaller scales:

\begin{lemma}\label{lower-bound-on-entropy-of-projected-components}
Assume that $\{\overline{A}_i\}_{i\in\Lambda}$ generates a strongly irreducible and unbounded subgroup of $GL_2(\R)$. Then for every $\varepsilon>0$, $m\geq M(\varepsilon)\geq1$ and $n\geq N(\varepsilon,m)$
\begin{equation}\label{eq:lower-bound-on-entropy-of-projected-components}
\inf_{V\in\RP}\PP\left(\alpha-\varepsilon\leq\frac{1}{m}H(\pi_V\varphi_{\UU(n)}\mu,\DD_{m-\log\|A_{\UU(n)}\|})\leq\alpha+\varepsilon\right)>1-\varepsilon.
\end{equation}

The same statement holds with $\DD_{m-\log\|A_{\UU(n)}^*|V\|}$ instead of  $\DD_{m-\log\|A_{\UU(n)}\|}$.
\end{lemma}

\begin{proof}
  The versions using $\DD_{m-\log\|A_{\UU(n)}^*|V\|}$ and $\DD_{m-\log\|A_{\UU(n)}^*\|}$ are equivalent because of the previous lemma (changing the level of the partition by an additive constant results in an $O(1/m)$ change to the entropy, which can be absorbed in $\varepsilon$). We prove the version with $\DD_{m-\log\|A_{\UU(n)}^*|V\|}$.

First, consider the entropy in the statement. Re-scaling the measure and partition by $\|A^*_{\UU(n)} | V\|$, and then applying a translation, causes the entropy to change by $O(1/m)$. Thus the statement is formally unchanged if we allow such a transformation. Therefore, using the identity \eqref{eq:affine-map-to-projection}, it is enough to show that for large enough $m$,\[
\inf_{V\in\RP}\PP\left(\alpha-\varepsilon\leq\frac{1}{m}H(\pi_{A_{\UU(n)^*V}}\mu,\DD_{m})\leq\alpha+\varepsilon\right)>1-\varepsilon.
  \]
  This would be an immediate consequence of the equidistribution of the random walk on $\RP$ (Proposition \ref{lem:uniform-convergence-to-furst-measure}) if the scale-$m$ entropy were continuous as a function of the measure. It is not, but as discussed at the end of Section \ref{sub:entropy}, one can replace entropy by a continuous variant at the cost of $O(1/m)$, and the resulting $O(1/m)$ error can again be absorbed in $\varepsilon$. This proves the claim.
\end{proof}

Finally, we make a similar statement for projections of the random cylinders $\pi_V\varphi_{\II(n)}\mu$ (note that the cylinder is chosen according to $\II(n)$ instead of $\UU(n)$ as in the previous lemma).

\begin{proposition}\label{thm:lower-bound-on-entropy-of-projected-components-2}
Assume that $\{\overline{A}_i\}_{i\in\Lambda}$ generates a strongly irreducible and unbounded subgroup of $GL_2(\R)$. Then for every $\varepsilon>0$, $m\geq M(\varepsilon)$ and $n\geq N(\varepsilon,m)$,
\[
\inf_{V\in\RP}\PP_{1\leq i\leq n}\left(\alpha-\varepsilon\leq\frac{1}{m}H(\pi_V\varphi_{\II(i)}\mu,\DD_{i+m})\leq\alpha+\varepsilon\right)>1-\varepsilon.
\]
\end{proposition}
\begin{proof}
  Let \[
  \mathcal{U}_V=\left\{\ii\in\Lambda^*: \alpha-\varepsilon\leq\frac{1}{m}H(\pi_V\varphi_{\ii}\mu,\DD_{-\log\|A_{\ii}\|+m})\leq\alpha+\varepsilon\right\},
  \]
  and apply Lemma~\ref{lem:converting-Un-to-In-or-Jn} to it; the hypothesis of the lemma is satisfied by Lemma \ref{lower-bound-on-entropy-of-projected-components}.
%  For the second statement, note that for each $i$ we have the identity $\mu = \EE(\varphi_{\II(i)}\mu)$, from which we get \[
%    \frac{1}{m}H(\pi_V\mu,\QQ_{i+m}) \geq \frac{1}{m}\EE(H(\pi_V\varphi_{\II(i)}\mu,\QQ_i+m)
%    \]
%    By the first part of the proposition, on average over $1\leq i \leq n$, the expectation above is at least $\alpha-2\varepsilon$. The claim now follows by the  multiscale entropy formula, similar to \cite[Lemma 3.7]{Hochman2015}, or by Lemma \ref{lem:multiscale-entropy-action} below, applied with $\theta=\delta_{\pi_V}$.
\end{proof}

\begin{proposition}\label{lem:projections-are-porous-from-scale-1-to-n}
  Assume that $\{\overline{A}_i\}_{i\in\Lambda}$ generates a strongly irreducible and unbounded subgroup of $GL_2(\R)$. For every $\varepsilon>0$, $m\geq M(\varepsilon)$, for $\eta^*$-a.e. $V$ and $n\geq N(\varepsilon,m,V)$, the projection $\pi_V\mu$ is $(\alpha,\varepsilon,m)$-entropy porous from scale $1$ to $n$.
\end{proposition}

\begin{proof}
  Let $\varepsilon>0$ be given, fix for the moment $V\in\RP$ and write $\tau=\pi_V\mu$. We note that of the parameters introduced later only $n$ will depend on $V$.

  Choose $\delta$ corresponding to $\varepsilon$ as in Lemma \ref{lem:entropy-porosity-from-covering}.

  Since $\tau$ is uniformly continuous across scales, for $\ell$ large enough we have $\tau(I)\leq \delta\tau(J)$ for any concentric intervals $I,J$ with $I\subseteq J$ and $|I|<q^{-\ell}|J|$. Fix such an $\ell$.

  Apply Proposition \ref{thm:lower-bound-on-entropy-of-projected-components-2} with $\delta^2/2$ in the role of $\varepsilon$ there, and let $m,n$ be as required there. We can assume that $m$ is large enough relative to $\ell$ that $m-\ell$ has the same stated property. We thus know that \[
\PP_{1\leq i\leq n}\left(\alpha-\frac{\delta^2}{2}\leq\frac{1}{m}H(\pi_V\varphi_{\II(i)}\mu,\DD_{i+m-\ell})\leq\alpha+\frac{\delta^2}{2}\right)>1-\frac{\delta^2}{2},
\]
and since we can take $n$ large with respect to $\ell$ we can shift the index range by $\ell$, and the change is less that $\delta^2/2$. We get \[
\PP_{1\leq i\leq n}\left(\alpha-\frac{\delta^2}{2}\leq\frac{1}{m}H(\pi_V\varphi_{\II(i+\ell)}\mu,\DD_{i+m})\leq\alpha+\frac{\delta^2}{2}\right)>1-\delta^2.
\]
Applying Markov's inequality, we find that for a $(1-\delta)$-fraction of levels $1\leq k\leq n$, we have \[
\PP_{i=k}\left(\frac{1}{m}H(\pi_V\varphi_{\II(i+\ell)}\mu,\DD_{i+m})\geq\alpha-\delta\right)>1-\delta.
\]
Finally, assuming that $V$ is $\eta^*$-typical, we know that the entropy dimension is $\alpha$, hence $|\frac{1}{n}H(\tau,\QQ_n) - \alpha|<\delta$ for large $n$. Entropy porosity now follows by applying Lemma \ref{lem:entropy-porosity-from-covering}, taking  $\tau=\pi_V\mu$, and for each scale $i$ taking $\tau_j$ to be the components in the event in the last equation with their natural probabilities, and $\tau_0$ the remaining mass of $\tau$. 

\end{proof}

\begin{corollary} Under the assumptions of the last proposition, for every $\ep>0$, $m>M(\ep)$ and $n>N(\ep,m)$, there is an open set $E=E(\ep,m,n)\subseteq\RP$ of measure $\eta^*(E)>1-\ep$ such that for every $V\in E$ the measure $\pi_V\mu$ is $(\alpha,\ep,m)$-entropy porous from scale $1$ to $n$.
\end{corollary}
\begin{proof}
  The existence of a measurable set $E$ as above follows directly from the previous  proposition, assuming $m,n$ are large as specified there. We would like to find such a set $E$ that is open. For any fixed $V\in\RP$ and $n$, observe that as $V'\to V$ we have $(\pi{V'}\mu)(I)\to (\pi_V\mu)(I)$ for every interval $I$, and also $(\pi_{V'}\mu)_I\to\pi_V\mu_I$ (we use here the non-atomicity of $\pi_V\mu$). It follows that as $V'\to V$, the  components distribution of $\pi_{V'}\mu$  at levels $1\leq i \leq n$ converges weakly to the corresponding distribution of $\pi_V\mu$. Furthermore, by \eqref{eq:entropy-under-transformation-3} applied to the maps $\pi_V,\pi_{V'}$,  if $V'$ is close enough to $V$ in a manner depending on $n$ but not on $V$, then the same scale $i+m$ entropy of corresponding level-$i$ components of $\pi_V\mu,\pi_{V'}\mu$ agree up to $O(1)$, which, after dividing by $m$, and assuming $m$ large enough, is an error less than $\ep$. All of this implies that if $m$ is large enough and $V\in E$, then for $V'$ close enough to $V$ we have that $\pi_{V'}\mu$ is  $(\alpha ,2\ep,m)$-entropy porous. Starting from $\ep/2$ instead of $\ep$, we have the claim.
\end{proof}

\begin{proposition}
  Assume that $\{\overline{A}_i\}_{i\in\Lambda}$ generates a strongly irreducible and unbounded subgroup of $GL_2(\R)$. For every $\varepsilon>0$, $m\geq M(\varepsilon)$, $k\geq K(\varepsilon,m)$ and $n\geq N(\varepsilon,m,k)$,
\begin{equation}\label{eq:projections-are-porous-1-to-n}
\inf_{V\in\RP}\PP_{1\leq i\leq n}\left(\substack{\text{\large $\pi_V\varphi_{\II(i)}\mu$ is $(\alpha,\varepsilon,m)$-entropy porous}\\ \text{\large from scale $i$ to $i+k$}}\right)>1-\varepsilon.
\end{equation}

\end{proposition}

\begin{proof}
  For any $V\in\RP$ and $\ii\in\Lambda^*$, the measure $\pi_V\varphi_\ii\mu$ is an affine image of $\pi_W\mu$ where $W=A_{\ii}^*V$, and the affine map scales by $\|A_\ii^* |V \|$. It follows from the previous corollary, from invariance of porosity under coordinate changes (Lemma \ref{lem:porosity-invarance-under-affine-maps}), and from the equidistribution of Proposition \ref{prop:convergence-to-furstenberg-measure}, that for every $\varepsilon>0$, for large enough $m$ and all large enough $k$, all but an arbitrarily small measure of the projected cylinders $\pi_V\varphi_\ii\mu$ are $(\alpha,\varepsilon,m)$-entropy porous from scale $\log\|A_\ii^* |V\|$ to  $\log\|A_\ii^* |V\|+k$.

  This is almost the conclusion we want, but we want porosity not at the scales given above, but rather at scales   $\log\|A_\ii^*\|$ to $\log\|A_\ii^*\|+k$ (because in the definition of $\II(i)$, the maps $\varphi_{\II(i)}$ are chosen so that $\|A_{\II(i)}\|\approx q^{-i}$). But by Lemma \ref{lem:bounding-norms-on-subspaces}, these ranges differ by an additive constant $c$ with high probability  between scales $1\leq k \leq n$, and the  probability can be made arbitrarily close to $1$  if we allow $c$ to be large. Clearly, if a measure is $(\alpha, \varepsilon,m)$-entropy porous at scales $i$ to $i+k$, then it is $(\alpha,\varepsilon+O(c/k),m)$-entropy porous from scales $i\pm c$ to $i+k\pm c$. Thus, by slightly reducing $\varepsilon$ to begin with, we have proved the proposition.
\end{proof}

\begin{proof}[Proof of Proposition \ref{prop:projections-are-porous}]
 Every $\psi\in \AAA_{2,1}$ has the form $x\mapsto r\pi_V(x)+a$ for some $V$. The claim follows formally from the previous proposition combined with Lemma \ref{lem:porosity-invarance-under-affine-maps}.
\end{proof}

\section{\label{sec:entropy-growth-under-convolutions}Entropy growth
of convolutions}

\subsection{\label{sub:convolution-Euclidean-case}Euclidean case}

Denote by $*$ the convolution of measures on $\mathbb{R}$. That is, for any $\theta,\tau$ Borel probability measures on $\R$
$$
\int f(x)d(\tau*\theta)(x)=\iint f(y+z)d\tau(y)d\theta(z),
$$
for any integrable function  $f$.

The entropy of a convolution is generally at least as large as each of
the convolved measures, although due to the discretization involved
there may be a small loss: For every boundedly supported $\tau,\theta\in\mathcal{P}(\R)$,
\[
\frac{1}{n}H(\tau,\DD_{n})-O(\frac{1}{n})\leq\frac{1}{n}H(\theta*\tau,\DD_{n})\leq\frac{1}{n}H(\tau,\DD_{n})+\frac{1}{n}H(\theta,\DD_{n})+O(\frac{1}{n}).
\]
(the error depends on the diameter of the supports; see \cite[Corollary 4.10]{Hochman2014}).
Typically, one expects that $\frac{1}{n}H(\theta*\tau,\DD_{n})$ is close to the upper bound, but
in general this is not the case, and one cannot rule out that the lower
bound is achieved, i.e. there is no entropy growth at all (in fact it is quite
non-trivial to give useful conditions under which the upper bound
is achieved). The following theorem, which follows directly from \cite[Theorem~2.8]{Hochman2014}, provides
a verifiable condition under which at least some entropy growth occurs.\footnote{Theorem 2.8 of \cite{Hochman2014} contains a slight error. The condition $\frac{1}{m}H(\nu,\mathcal{D}_n)>\varepsilon$ there should be replaced by $>2\varepsilon$, or by $>c\varepsilon$., where $c$ is any constant larger than one, but then the dependence of $\delta$ and the other parameters on $\varepsilon$ depend on $c$. Also, the theorem is stated for dyadic partitions rather than $q$-adic, but this modification is harmless.} 

\begin{theorem}\label{thm:entropy-growth-under-convolution}
For every $\varepsilon>0$ and $m\ge1$ there exists $\delta=\delta(\varepsilon,m)>0$,
such that all $n>N(\varepsilon,\delta,m)$ the following holds.

Let $k\ge0$ and $\tau,\theta\in\mathcal{P}(\mathbb{R})$, and suppose that
\begin{enumerate}
  \item $\tau,\theta$ are supported on intervals of length $q^{-k}$. 
  \item $\tau$ is $(1-\varepsilon,\varepsilon/2,m)$-entropy porous from scale $k$ to $k+n$.
  \item $\frac{1}{n}H(\theta,\DD_{k+n})>\varepsilon$.
\end{enumerate}
Then
\[
\frac{1}{n}H(\theta*\tau,\DD_{k+n})\ge\frac{1}{n}H(\tau,\DD_{k+n})+\delta.
\]
\end{theorem}

We could replace condition (1) of the theorem by the assumption that the measures are supported on sets of diameter $c\cdot q^{-k}$, if $\delta$ and $n$ are allowed to depending on $c$. This can be done by replacing $k$ with $k-\log c$, which requires one to adjust $\ep$ bu $O_c(1/n)$ in (2) and (3), and is compensated for by the adjusting of $\delta$ and $n$.

\subsection{\label{sub:Convolution-with-measures-on-projections}Convolution
with measures on $\AAA_{2,1}$}

Let $f\colon \AAA_{2,1}\times\R^2\mapsto\R$ be the natural action map. Namely,
$$
f(\psi,x)=\psi(x)\text{ for $\psi\in \AAA_{2,1}$ and $x\in\R^2$}.
$$
For measures $\theta\in\mathcal{P}(\AAA_{2,1})$ and $\tau\in\mathcal{P}(\mathbb{R}^{2})$,
let $\theta.\tau\in\mathcal{P}(\R)$ denote the push-forward measure of $\theta\times\tau$ via the map $f$. That is,
$$
\int g(x)d(\theta.\tau)(x)=\iint g(\psi y)d\theta(\psi)d\tau(y).
$$

Since the map $f$ is not linear, Theorem \ref{thm:entropy-growth-under-convolution} and its multi-dimensional analogues do not apply to the convolution operation $\boldsymbol{.}$. But $f$ is a smooth map, so
at small scales is approximately linear. Fixing any $0<c<1$, if $r$ is small enough, then the action of the map on an $r$-ball,
viewed at resolution $cr$, is very close to the action of its derivative on that ball (or rather on the lift of the ball to the tangent space). This is the idea behind the linearization technique
from \cite[Section 5.6]{Hochman2015}. We state and prove here a special
case of this method, adapted to our setting.

\begin{lemma}\label{lem:linearization}
Let $Z\subseteq \AAA_{2,1}\times\mathbb{R}^{2}$ be a compact set. Then for every $\varepsilon>0$, $k>K(\varepsilon)$ and $0<\rho<\rho(Z,\varepsilon,k)$, the following holds.

For any $(\psi_{0},x_{0})\in Z$ and for any $\theta\in\mathcal{P}(B(\psi_{0},\rho))$, $\tau\in\mathcal{P}(B(x_{0},\rho))$,
\[
\left|\frac{1}{k}H(\theta\boldsymbol{.}\tau,\DD_{k-\log\rho})-\frac{1}{k}H((\theta\boldsymbol{.}x_{0})*(\psi_{0}\tau),\DD_{k-\log\rho})\right|<\varepsilon\:.
\]
\end{lemma}

\begin{proof}
Identify $\AAA_{2,1}$ with $\mathbb{R}^{3}$ so that $(a,b,c)\in\mathbb{R}^{3}$
corresponds to the map $\psi\in \AAA_{2,1}$,
\[
\psi(u,v) = au+bv+c.
\]
(this differs from the parametrization in Section~\ref{sub:affine-maps-and-invariant-metric}, but this does not affect the argument). Thus, $f$ can be identified with a map $f\colon\R^5\mapsto\R$ that $f(a,b,c,u,v)=au+bv+c$. For a $z_{0}=(\psi_{0},x_{0})\in Z$, write
\[
df_{z_{0}}=(A_{z_{0}}\quad B_{z_{0}}),
\]
where $A_{z_{0}}$ is a $1\times3$ real matrix and $B_{z_{0}}$ is
$1\times2$. An elementary calculation shows that if $\psi_{0}(u,v)=a_0u+b_0v+c_0$
and $x_{0}=(s_0,t_0)$ then $A_{z_{0}}=(s_0,t_0,1)$ and $B_{z_{0}}=(a_0,b_0)$.
Thus if we write
\[
h_{z_{0}}(z)=f(z_{0})+df_{z_{0}}(z-z_{0})
\]
for the first-order approximation of $f$ at $z_{0}$, then for $z=(\psi,x)$
we have
\[
h_{z_{0}}(\psi,x)=\psi(x_{0})+\psi_{0}(x)-\psi_{0}(x_{0}).
\]
This gives,
\[
(h_{z_{0}})_*(\theta\times\tau)=\delta_{-\psi_{0}(x_{0})}*(\theta\boldsymbol{.}x_{0})*(\psi_{0}\tau).
\]

Note that the norm metric on $\AAA_{2,1}$, induced by the euclidean
norm on $\mathbb{R}^{3}$, and the invariant metric $d$ on $\AAA_{2,1}$,
are bi-Lipschitz equivalent on compact sets. Hence, by compactness
of $Z$, for every $k\geq1$ there exists $\rho_0=\rho_0(Z,k)$ such that for every $0<\rho<\rho_0$ and every $z_{0}=(\psi_{0},x_{0})\in Z$,
\[
\left\Vert f-h_{z_{0}}\right\Vert _{C(B(\psi_{0},\rho)\times B(x_{0},\rho))}<\rho q^{-k},
\]
where the norm $\|\cdot\|_E$ denotes the supremum norm on the domain $E$. Then by \eqref{eq:entropy-under-transformation-3}
\begin{eqnarray*}
H(f(\theta\times\tau),\DD_{k-\log\rho}) & = & H((h_{z_{0}})_*(\theta\times\tau),\DD_{k-\log\rho})+O(1)\\
 & = & H((\theta\boldsymbol{.}x_{0})*(\psi_{0}\tau),\DD_{k-\log\rho})+O(1).
\end{eqnarray*}
Dividing by $k$, the error term $O(1/k)$ can be made less than $\varepsilon$ by taking $k$ large, which is what we wanted to prove.
\end{proof}

Recall that for an affine transformation $\psi\in \AAA_{2,1}$, $\|\psi\|$ denotes the induced norm of the linear operator $x\mapsto\psi(x)-\psi(0)$. Also recall that $\mu$ denotes the self-affine measure.

It will be convenient to define a random measure $\mu_{\varphi(x,n)}$, and implicitly a random point $x$, in analogy to the component $\mu_{x,n}$. This is just the random measure $\varphi_{\II(n)}\mu$ together with a point $x$ chosen with distribution $\varphi_{\II(n)}\mu$ (conditionally independently of the choice of $\II(n)$). When $\mu_{\varphi(x,n)}$ appears in probabilistic settings, our conventions are the same as for random components. This notation will be only used for the random word $\II(n)$.

\begin{lemma}
\label{lem:multiscale-entropy-action}Let $t>0$ and let $\theta\in\mathcal{P}(\AAA_{2,1})$ 
satisfy $\Vert\psi\Vert=\Theta(q^{-t})$ for all $\psi\in\supp(\theta)$.
Then for every $1\leq k \leq n$,
\[
\frac{1}{n}H(\theta\boldsymbol{.}\mu,\DD_{t+n})\geq\mathbb{E}_{1\leq i\leq n}\left(\frac{1}{k}H(\theta_{\psi,i}\boldsymbol{.}\mu_{\varphi(i,x)},\DD_{t+i+k})\right)-O(\frac{k}{n}+\frac{1}{k})\:.
\]
\end{lemma}

\begin{proof}
  Since for any pair of component $\theta_{\psi,i}$ and cylinder $\mu_{\varphi(i,x)}$, by definition we have (a) $\psi\in\supp\theta_{\psi,i}$ and $x\in\supp\mu_{\varphi(i,x)}$, (b) the components  are supported on sets of diameter $O(q^{-i})$, and (c) $\supp(\theta_{\psi,i}\conv\mu _{\varphi(i,x)}\subseteq (\supp\theta_{\psi,i})\conv(\supp\mu_{\varphi(i,x)})$, by Lemma \ref{lem:diameter-bound} and the hypothesis on $\|\psi\|$ we have \[
  \diam(\theta_{\psi,i}\boldsymbol{.}\mu_{\varphi(i,x)})=O(\|\psi\|q^{-i}) = O(q^{-(i+t)}).
  \]

Let $\ell$ be the integral part of $\frac{n}{k}$. From the identity
\[
\theta\boldsymbol{.}\mu=\mathbb{E}_{i}\left(\theta_{\psi,i}\boldsymbol{.}\mu_{\varphi(i,x)}\right),
\]
(which is valid for each $i\ge 1$), we get that for every residue $0\leq r < k$,

\begin{equation}\label{eq:4}
\begin{split}
  H(\theta\boldsymbol{.}\mu,\DD_{t+n})&=\sum_{m=0}^{\ell-2}H\left(\theta\boldsymbol{.}\mu,\DD_{t+(m+1)k+r}\mid\DD_{t+mk+r}\right)\\
  &\;\;\;\;+H(\theta\conv\mu,\QQ_{t+r})+H(\theta\conv\mu,\QQ_{t+n}|\QQ_{t+(\ell-1)k+r})\\
  & \geq  \sum_{m=0}^{\ell-2}H\left(\theta\boldsymbol{.}\mu,\DD_{t+(m+1)k+r}\mid\DD_{t+mk+r}\right)\\
&\geq \sum_{m=0}^{\ell-2}\mathbb{E}_{i=mk+r}\left(H\left(\theta_{\psi,i}\boldsymbol{.}\mu_{\varphi(i,x)},\DD_{t+k+i}\mid\DD_{t+i}\right)\right)\:-O(1).
\end{split}
\end{equation}
where in the last line we used concavity of entropy to deal with the main term. Also, since $\theta_{\psi,i}\boldsymbol{.}\mu_{\varphi(i,x)}$ is supported on a set of diameter $O(q^{-(t+mk+r)})$ (since  $i=mk+r$), we can dispose of the conditioning in \eqref{eq:4} at the cost of an $O(1)$ error in each summand, and obtain
\[
H(\theta\boldsymbol{.}\mu,\DD_{t+n})\geq\sum_{m=0}^{\ell-2}\mathbb{E}_{i=mk+r}\left(H\left(\theta_{\psi,i}\boldsymbol{.}\mu_{\varphi(i,x)},\DD_{t+k+i}\right)\right)+O(\ell)\:.
\]
Now by averaging over $0\le r\le k-1$ and dividing by $n$, and recalling that $\ell/n\leq1/k$, we get,
\begin{multline*}
\frac{1}{n}H(\theta\boldsymbol{.}\mu,\DD_{t+n})\geq\frac{1}{n}\sum_{r=0}^{k-1}\sum_{m=0}^{\ell-2}\mathbb{E}_{i=mk+r}\left(\frac{1}{k}H\left(\theta_{\psi,i}\boldsymbol{.}\mu_{\varphi(i,x)},\DD_{t+k+i}\right)\right)-O(\frac{\ell}{n})\\
=\mathbb{E}_{1\leq i\leq n}\left(\frac{1}{k}H\left(\theta_{\psi,i}\boldsymbol{.}\mu_{\varphi(i,x)},\DD_{t+k+i}\right)\right)-O(\frac{k}{n}+\frac{1}{k}),
\end{multline*}
which completes the proof of the lemma.
\end{proof}
The proof of the following lemma is similar to the one above or to \cite[Lemma~3.4]{Hochman2014}, and so
it is omitted.
\begin{lemma}
\label{lem:multiscale-entropy-affine-space}Let $\theta\in\mathcal{P}(\AAA_{2,1})$
and $n\ge k\ge1$ be given. Set $r=\diam(\supp(\theta))$, then
\[
\frac{1}{n}H(\theta,\DD_{n})=\mathbb{E}_{1\le i\le n}\left(\frac{1}{k}H\left(\theta_{\psi,i},\DD_{i+k}\right)\right)+O(\frac{k}{n}+\frac{\log(1+r)}{n})\:.
\]
\end{lemma}

 The proof of the following lemma is similar to the one given in \cite[Corollary~5.10]{Hochman2015}, but for completeness we give the details.

\begin{lemma}
\label{lem:ent-on-plane-from-ent-on-aff-gr}For every compact set $Z\subseteq \AAA_{2,1}$ there exists a constant
$C=C(Z,\mu)\ge1$ such that for every $\theta\in\mathcal{P}(\AAA_{2,1})$ supported on $Z$ and every
$k,i\ge1$,
\[
\mu\{x\::\:\frac{1}{k}H(\theta\boldsymbol{.}x,\DD_{i+k})\ge\frac{1}{Ck}H(\theta,\DD_{i+k}^{\AAA_{2,1}})-\frac{C}{k}\}\ge C^{-1}\:.
\]
\end{lemma}

\begin{proof}
For a subset $W$ of a metric space write $W^{(\ep)}$ for its $\ep$-neighborhood. By Lemma~\ref{lem:uniform-continuity-across-scales}, for every $\varepsilon>0$ there exists a $\rho>0$ such that for every hyperplane $W\subseteq\R^2$, $\mu(W^{(\ep)})<\varepsilon$. Thus, similarly to \cite[Lemma~5.9]{Hochman2015}, for every Borel set $A\subseteq\R^2$ with $\mu(A)>(3\varepsilon)^{1/3}$ there exist $x_1,x_2,x_3\in A$ such that the distance between $x_i$ and the hyperplane defined by the two points $\{x_1,x_2,x_3\}\setminus\{x_i\}$ is at least $\rho$. We say in this case that the points $\{x_1,x_2,x_3\}$ is $\rho$-independent.

\begin{claim}
There exists a constant $C'=C'(Z,\rho)$ such that if $\{x_1,x_2,x_3\}$ are $\rho$-independent points, then $$H(\theta,\DD_{i+k}^{\AAA_{2,1}})\leq\sum_{i=1}^3H(\theta.x_i,\DD_{i+k})+C'.$$
\end{claim}
\begin{proof}
  Let $\pi_i:\R^3\mapsto\R$ be the coordinate projections. Define $g:\AAA_{2,1}\mapsto\R^3$ function as $g(\psi)=(\psi(x_1),\psi(x_2),\psi(x_3))$. It is easy to see that $g$ is a diffeomorphism and that its restriction to $Z$ is bi-Lipschitz to its image g(Z) with constants which depend only on $Z$ and $\rho$.  Thus, % \eqref{eq:entropy-under-transformation-2}
there is a constant $C'$ depending on $Z,\rho$ such that 
$$
\left|H(\theta,\DD_{i+k}^{\AAA_{2,1}})-H(f\theta,\QQ_{i+k}^{\R^3})\right|\leq C'.
$$
 Since $\QQ_{i+k}^{\R^3}=\bigvee_{j=1}^3\pi_j^{-1}\QQ_{i+k}^\R$, this is the same as
$$
\left|H(\theta,\DD_{i+k}^{\AAA_{2,1}})-H(f\theta,\bigvee_{j=1}^3\pi_i^{-1}\QQ_{i+k})\right|\leq C'.
$$ 
The statement now follows by
$$
H(f\theta,\bigvee_{i=1}^3\pi_i^{-1}\DD_n)\leq\sum_{i=1}^3H(f\theta,\pi_i^{-1}\DD_n)=\sum_{i=1}^3H(\theta.x_i,\DD_n).
$$
\end{proof}
Let $C'$ be as in the claim and set
$$
B=\left\{x\in\R^2:\frac{1}{k}H(\theta.x,\DD_{i+k})>\frac{1}{3k}H(\theta,\DD_{i+k}^{\AAA_{2,1}})-\frac{C'}{k})\right\}.
$$
\begin{claim}
 $\mu\left(B\right)>1-(3\varepsilon)^{1/3}$.
\end{claim}
\begin{proof}
We argue by contradiction. Suppose that $\mu(\R^2\setminus B)>(3\varepsilon)^{1/3}$. Thus, there exist $\rho$-independent points $x_1,x_2,x_3\in\R^2\setminus B$. Hence, by applying the claim
$$
H(\theta,\DD_{i+k}^{\AAA_{2,1}})\leq\sum_{i=1}^3H(\theta.x_i,\DD_{i+k})+C'\leq H(\theta,\DD_{i+k}^{\AAA_{2,1}})-2C',
$$
which is a contradiction.
\end{proof}

The lemma follows by choosing $\varepsilon=1/81$ and $C=\max\{3,C'\}$.

\end{proof}

\begin{theorem}
\label{thm:inverse-theorem-for-projections} Assume that $\{\overline{A}_i\}$ generates a non-compact and totally irreducible subgroup of $GL_2(\R)$. Let $\alpha$ be as defined
in Lemma~\ref{lem:alpha} and assume
that $\alpha<1$. Then for every $\varepsilon>0$ there exists $\delta=\delta(\varepsilon)>0$ such that for $n\ge N(\varepsilon,\delta)$ the following holds:

Let
$t\ge1$ and let $\theta\in\mathcal{P}(\AAA_{2,1})$ satisfy $\Vert\psi\Vert=\Theta(q^{-t})$
for all $\psi\in\supp(\theta)$. Assume that $\diam(\supp(\theta))<1/\varepsilon$
and $\frac{1}{n}H(\theta,\DD_{n})>\varepsilon$, then $\frac{1}{n}H(\theta\boldsymbol{.}\mu,\DD_{t+n})>\alpha+\delta$.
\end{theorem}

\begin{remark}
The proof can be adapted to other measures than $\mu$. In the proof we use the cylinder structure of $\mu$ but can be adapted to any measure $\mu$ such that $\psi(\mu_{x,i})$ is entropy porous, at suitable scales, for most components $\mu_{x,i}$ and all $\psi\in \AAA_{2,1}$.
\end{remark}

\begin{proof}
Let $\varepsilon>0$, and let $n\ge k\ge1$ be integers such that $\frac{1}{k}$
and $\frac{k}{n}$ are small in a manner described later in the proof.
In particular, these quantities are small with respect to $\varepsilon$.
Let $\theta\in\mathcal{P}(\AAA_{2,1})$ and  $t\ge1$ be as stated.
By the assumption $\frac{1}{n}H(\theta,\DD_{n})>\varepsilon$ and Lemma \ref{lem:multiscale-entropy-affine-space},
\begin{equation}
\mathbb{E}_{1\le i\le n}\left(\frac{1}{k}H\left(\theta_{\psi,i},\DD_{i+k}\right)\right)>\varepsilon-O(\frac{k+\log(\ep^{-1})}{n})>\frac{\varepsilon}{2}\:.\label{eq:6}
\end{equation}

\begin{claim} There exists a constant $C$ such that
  \begin{equation}
\mathbb{P}_{1\le i\le n\,,\,x\sim\mu}\left(\frac{1}{k}H\left(\theta_{\psi,i}\boldsymbol{.}x,\DD_{i+k+t}\right)>\frac{\varepsilon}{C}\right)>\frac{\varepsilon}{C}\:.\label{eq:7}
\end{equation}
\end{claim}

\begin{proof}
By Lemma \ref{lem:Qn-has-bounded-degree} and the trivial bound on entropy in terms of partition size, the integrand on the left hand side of \eqref{eq:6}
is $O(1)$, and so for some global constant $C_{1}>1$,
\[
\mathbb{P}_{1\le i\le n}\left(\frac{1}{k}H\left(\theta_{\psi,i},\DD_{i+k}\right)>\frac{\varepsilon}{3}\right)>\frac{\varepsilon}{C_{1}}\:.
\]
Thus, it follows from the invariance of the metric (specifically, Lemma \ref{lem:Qn-has-bounded-degree}, and the fact that each $q$-adic cell is bounded within and without by a ball of comparable diameter), and the fact that $k$ is large relative to $\varepsilon$, that
\[
\mathbb{P}_{1\le i\le n}\left(\frac{1}{k}H\left(S_{q^{t}}T_{-\psi(0)}\theta_{\psi,i},\DD_{i+k}\right)>\frac{\varepsilon}{4}\right)>\frac{\varepsilon}{C_{1}}\:.
\]
As in e.g. the proof of Lemma \ref{lem:diameter-bound}, the maps $S_{q^{t}}T_{-\psi(0)}\theta_{\psi,i}$ lie in a compact set in $\AAA_{2,1}$, hence by Lemma \ref{lem:ent-on-plane-from-ent-on-aff-gr} there exists a global constant $C>1$ with
\[
\mathbb{P}_{1\le i\le n\,,\,x\sim\mu}\left(\frac{1}{k}H\left(S_{q^{t}}T_{-\psi(0)}\theta_{\psi,i}\boldsymbol{.}x,\DD_{i+k}\right)>\frac{\varepsilon}{C}\right)>\frac{\varepsilon}{C}\:.\label{eq:7a}
\]
Taking into account how scaling and translation affect entropy, this is the same (up to an error that is absorbed in the constants) as \eqref{eq:7}.
\end{proof}

Let $C$ be as in the claim and set
\[
\varepsilon'=\min\{\frac{\varepsilon}{C},\frac{1-\alpha}{2}\}.
\]
Let $m\ge 1$. By assuming $m$ is large enough,
in a manner depending only on $\alpha,\ep'$,
and by assuming $k,n$ are large enough, in a manner depending on
all parameters, we get from Proposition \ref{prop:projections-are-porous}
that for each $\psi\in\supp(\theta)$,
\begin{equation}
\mathbb{P}_{1\leq i\leq n}\left(\begin{array}{c}
\psi\mu_{\varphi(i,x)}\mbox{ is }(\alpha,\frac{\ep'}{2},m)\text{-entropy}\\
\text{porous from scale \ensuremath{t+i} to \ensuremath{t+i+k}}
\end{array}\right)>1-\frac{\ep'}{2}\:.\label{eq:8}
\end{equation}
Fix a parameter $\rho>0$, which will later be taken small with respect to $\varepsilon,m$.  From Proposition \ref{thm:lower-bound-on-entropy-of-projected-components-2}, and by assuming
$k,n$ are large enough, we get that for each $\psi\in\supp(\theta)$,
\begin{equation}
\mathbb{P}_{1\leq i\leq n}\left(\frac{1}{k}H(\psi\mu_{\varphi(i,x)},\DD_{t+i+k})\ge\alpha-\rho\right)>1-\rho\:.\label{eq:9}
\end{equation}
Let $\delta=\delta(\varepsilon',m)>0$ be as obtained from Theorem \ref{thm:entropy-growth-under-convolution}.
Note that $\delta$ does not depend on $\rho$, hence we can assume
that $\rho$ is small with respect to $\varepsilon$ and $\delta$.

By Lemma \ref{lem:multiscale-entropy-action} we get,
\begin{multline}
\frac{1}{n}H(\theta\boldsymbol{.}\mu,\DD_{t+n})\ge\mathbb{E}_{1\leq i\leq n}\left(\frac{1}{k}H(\theta_{\psi,i}\boldsymbol{.}\mu_{\varphi(i,x)},\DD_{t+i+k})\right)-O(\frac{k}{n}+\frac{1}{k})\\
=\mathbb{E}_{1\leq i\leq n}\left(\frac{1}{k}H((S_{q^t}T_{-\psi(0)}\theta_{\psi,i})\boldsymbol{.}\mu_{\varphi(i,x)},\DD_{i+k})\right)-O(\frac{k}{n}+\frac{1}{k}),\label{eq:5}
\end{multline}
where we emphasize that the ``re-scaling'' is chosen so that $S_{q^{t}}T_{-\psi(0)}\psi$ is a linear map, and since $\|\psi\|=\Theta(q^{-t})$, this map lies in a fixed compact set in $\AAA_{2,1}$. Hence,
by assuming that $\frac{1}{k}$ and $\frac{k}{n}$ are sufficiently
small and by applying Lemma \ref{lem:linearization} to most $1\le i\le n$,
we get from \eqref{eq:5},
\begin{multline}
  \frac{1}{n}H(\theta\boldsymbol{.}\mu,\DD_{t+n})\\
  \ge\mathbb{E}_{1\leq i\leq n}\left(\frac{1}{k}H((S_{q^t}T_{-\psi(0)}\theta_{\psi,i}\boldsymbol{.}x)*(S_{q^t}T_{-\psi(0)}\psi\mu_{\varphi(i,x)}),\DD_{i+k})\right)-\rho/2\\
    \geq\mathbb{E}_{1\leq i\leq n}\left(\frac{1}{k}H((\theta_{\psi,i}\boldsymbol{.}x)*(\psi\mu_{\varphi(i,x)}),\DD_{i+k+t})\right)-\rho\:.
    \label{eq:endgame}
\end{multline}

Note that, by assuming $k$ is sufficiently large with respect to
$\rho$, the integrand in the last expectation is at least
\begin{equation}
\frac{1}{k}H(\psi\mu_{\varphi(i,x)},\DD_{i+k+t})-\rho\:,\label{eq:14}
\end{equation}
and by \eqref{eq:9}, outside an event of probability $<\rho$ over the cylinders $\mu_{\varphi(i,x)}$, $1\leq i \leq n$, this lower bound  is itself bounded below by $\alpha-2\rho$. Thus, if we write \begin{equation}
  p = \PP_{1\leq i \leq n}\left( \frac{1}{k}H((\theta_{\psi,i}\boldsymbol{.}x)*(\psi\mu_{\varphi(i,x)}),\DD_{i+k+t}) > \alpha + \delta\right),\label{eq:13}
  \end{equation}
  we get the lower bound
  \begin{eqnarray}
    \frac{1}{n}H(\theta\boldsymbol{.}\mu,\DD_{t+n}) & \geq & (\alpha+\delta)p + (1-p-\rho)(\alpha-2\rho)\\
    & = & \alpha + \delta p - \rho(\alpha + 2(1-p-\rho))\\
    & \ge & \alpha + \delta p - 3\rho. \label{eq:15}
  \end{eqnarray}
  (the probability $1-p-\rho$ comes from intersecting the event in \eqref{eq:9}, whose whose probablity is $>1-\rho$, with the complement of the event in \eqref{eq:13}, whose probability is $1-p$).
    
To complete the proof we estimate $p$ by applying Theorem \ref{thm:entropy-growth-under-convolution}. For this we can assume that $k$ is large with respect to $m$ and $\delta$.
Observe that for $x\in\supp(\mu)$ and $\psi\in\supp(\theta)$,
\[
\diam(\supp(\theta_{\psi,i}\boldsymbol{.}x))\text{ and }\diam(\supp(\psi\mu_{\varphi(i,x)}))\leq O(q^{-i-t})\:.
\]
so (by the hypotheses of Theorem \ref{thm:entropy-growth-under-convolution}), the event in \eqref{eq:13} is contained in the intersection of the events in \eqref{eq:7} and \eqref{eq:8}, and so \[
  p>\frac{\ep}{C}-\frac{\ep'}{2}>\frac{\ep}{2C}\:.
  \]
Inserting this into the lower bound \eqref{eq:15}, and assuming $\rho$ is sufficiently small with respect to $\varepsilon$ and $\delta$, the theorem is proved.
\end{proof}

\section{\label{sec:Separation}Separation}

Let $D$ be the metric on $\AAA_{2,2}$ defined in \eqref{eq:metricD}, and  recall that it is bi-Lipschitz equivalent to the metric induced by the operator norm when $\AAA_{2,2}$ is identified in the standard way with a subgroup of $GL_3(\R)$. Recall that $\Phi=\{\varphi_i\}$ is exponentially separated if there is a constant $c>0$ such that $D(\varphi_\ii,\varphi_\jj)>c^n$ for every $n$ and distinct $\ii,\jj\in\Lambda^n$, and this property is independent of the the bi-Lipschitz equivalence class of $D(\cdot,\cdot)$. Note that since the $\varphi_i\in\Phi$ are contractions, by \eqref{eq:metricD} we have $D(\varphi_\ii\psi',\varphi_\ii\psi'')\leq\|\varphi_\ii\|D(\psi',\psi'')\leq D(\psi',\psi'')$ for all $\ii\in\Lambda^*$ and $\psi',\psi''\in\AAA_{2,2}$, and also, since $D(\varphi_\jj,\id)\leq C$ for some $C=C(\Phi)$ and all $\jj\in\Lambda^*$, we have $D(\psi'\varphi_\jj, \psi''\varphi_\jj)\leq CD(\psi',\psi'')$ for all $\jj\in\Lambda^*$.

 \begin{lemma}
	\label{lem:extended-sep}Let $\Phi=\{\varphi_{i}\}$
	be exponentially separated. Then there exists $b>0$ so that for every
	$\ii,\jj\in\Lambda^*$ with $|\ii|\geq |\jj|$ such that $\jj$ not a prefix of $\ii$, we have $D(\varphi_{\ii},\varphi_{\jj})\ge b^{|\ii|}$.
\end{lemma}

\begin{proof}
	Let $c$ denote the constant verifying exponential separation and suppose that the lemma were false. Let $0<b<1$ and let $\ii,\jj\in\Lambda^*$ with $n=|\ii|\geq |\jj|=m$, such that $\jj$ is not a prefix of $\ii$ and $D(\varphi_{\ii},\varphi_{\jj})<b^{n}$. Since $\jj$ is not a prefix of $\ii$ we have $\ii\jj\ne\jj\ii$, hence by the remarks before the lemma, for some $C=C(\Phi)$,
	\begin{eqnarray*}
	  c^{2n} &\leq &c^{n+m}\\
          &\le & D(\varphi_{\ii\jj},\varphi_{\jj\ii})\\
          &\le & D(\varphi_{\jj\ii},\varphi_{\jj\jj})+D(\varphi_{\jj\jj},\varphi_{\ii\jj})\\
          &<&2Cb^n.
	\end{eqnarray*}
	Since for every $0<b<1$ there exists an $n$ for which the above holds, we have arrived at a contradiction.
\end{proof}
Recall that we have is endowed  $\AAA_{2,1}$ with an $\AAA_{1,1}$-invariant metric $d$.

\begin{proposition}
	\label{prop:separation}Assume $\Phi=\{\varphi_{i}\}$
	is exponentially separated and let $0<s<1$. Then there is a $c=c(s)>0$
	and a set $ E= E(s)\subseteq \RP$ such that:
	\begin{enumerate}[(a)]
		\item $\dim E\le s$,
		\item For all $V\in\RP\setminus E$ and $n$ large
	enough, $\{\pi_V\varphi_{\ii}\}_{\ii\in\Xi^V_n}$ are $c^{n}$-separated in $\AAA_{2,1}$ with respect to $d$.
	\end{enumerate}
\end{proposition}

\begin{proof}
	For $n\ge1$ define
	\[
	\mathcal{U}_{n}=\left\{ (\ii,\jj)\in\Lambda^*\times\Lambda^*\::\:\begin{array}{c}
	|\ii|,|\jj|\le n,\:\ii\text{ is not a prefix of $\jj$,}\\
	\text{and $\jj$ is not a prefix of $\ii$}
	\end{array}\right\} \:.
	\]
	Let $0<b<1$ be as in Lemma \ref{lem:extended-sep}, so that  $D(\varphi_{\ii},\varphi_{\jj})\ge b^{n}$ for each $(\ii,\jj)\in\mathcal{U}_{n}$. For such $\ii,\jj$ there exists a vector $w^{n}_{\ii,\jj}\in \R^2$ with $\|w^{n}_{\ii,\jj}\|\leq 1$, satisfying
	\begin{equation}
	\|\varphi_{\ii}(w^{n}_{\ii,\jj})-\varphi_{\jj}(w^{n}_{\ii,\jj}))\|\ge b^{n}\:.\label{eq:16}
	\end{equation}
	Let $0<c<b/|\Lambda|^{2/s}$, and for each $n\ge1$ and $(\ii,\jj)\in\mathcal{U}_{n}$ set
	\[
	 E_{\ii,\jj}^{n}=\{V\in \RP\::\:d(\pi_V\varphi_{\ii}(w^{n}_{\ii,\jj}),\pi_V\varphi_{\jj}(w^{n}_{\ii,\jj}))<c^{n}\}\:.
	\]
	Now observe an elementary fact: there exists a constant $C>0$
	such that for each $x,x'\in\mathbb{R}^{2}$ and $\delta>0$, the set
	of $V\in \RP$ such that $d(\pi_Vx,\pi_Vx')<\delta$ is of diameter
	at most $C\delta/d(x,x')$. Applying this with $x=\varphi_\ii(w_{\ii,\jj}^n)$, $x'=\varphi_\jj(w^n_{\ii,\jj})$ and $\delta=c^n$, we find that \[
          \diam E_{\ii,\jj}^n < C(\frac{c}{b})^n\:.
        \]
	Next, define
	\[
	 E=\bigcap_{N=1}^{\infty}\bigcup_{n=N}^{\infty}\bigcup_{(\ii,\jj)\in\mathcal{U}_{n}} E_{\ii,\jj}^{n}.
	\]
	For each $n\ge1$ we have $|\mathcal{U}_{n}|\le |\Lambda|^{2(n+1)}$. Combining these facts with  $|\Lambda|^{2}c^{s}/b^{s}<1$ and the bound on the earlier diameter of $E^n_{\ii\jj}$, and writing  $\mathcal{H}_{\infty}^{s}$ for the $s$-dimensional Hausdorff content on $\RP$, we get
	\begin{eqnarray}
	  \mathcal{H}_{\infty}^{s}( E) & \le & \lim_{N\to\infty}\sum_{n=N}^{\infty}\:\sum_{(\ii,\jj)\in\mathcal{U}_{n}}(\diam\{ E_{\ii,\jj}^{n}\})^{s}\\
          &\le&\lim_{N\to\infty}\sum_{n=N}^{\infty}C^{s}|\Lambda|^{2(n+1)}\frac{c^{sn}}{b^{sn}}\\
          &=&0\;.
	\end{eqnarray}
	Thus, $\dim_H E\le s$.
	
	It remains to show that if $V\in\RP\setminus E$ and $\ii,\jj\in\Xi^V_n$ are distinct, then $\pi_V\varphi_\ii,\pi_V\varphi_\jj$ are exponentially far apart. Fixing $V\in\RP\setminus E$, it follows from the definition of $E$ that for every $V\in\RP\setminus E$ there exists $N=N(V)\geq1$ such that $D(\pi_V\varphi_{\ii},\pi_V\varphi_{\jj})\ge c^{n}$
	for all $n\ge N$ and $(\ii,\jj)\in\mathcal{U}_{n}$. This is almost what we want, except that we want to know it for $\ii,\jj\in\Xi^V_n$ instead of $\mathcal{U}_n$. To deduce this, observe that the $\varphi_i$ contract uniformly, so there exists a constant $\ell\geq1$ such that for every $n\geq1$
	\begin{equation}
	\Vert\pi_V\varphi_{\ii}\Vert\le q^{-n}\text{ for all }V\in \RP,\:k\ge \ell n\text{ and }\ii\in\Lambda^{k}\:.\label{eq:12}
	\end{equation}
So,  given $\ii,\jj\in\Xi_n^V$, if $\ii\ne\jj$ then $(\ii,\jj)\in\mathcal{U}_{n\ell}$. Hence, for every $n\ge N(V)$, $\{\pi_V \varphi_{\ii}\}_{\ii\in\Xi_n^V}$ are $c^{n\ell}$-separated in $\AAA_{2,1}$ with respect to $D$. Thus, by Lemma~\ref{lem:equivaence-of-metrics},
	\begin{align*}
	d(\pi_V \varphi_{\ii},\pi_V \varphi_{\jj})&=d(S_{q^n} \pi_V \varphi_{\ii},S_{q^n} \pi_V \varphi_{\jj})\\
	&\geq CD(S_{q^n} \pi_V \varphi_{\ii},S_{q^n} \pi_V \varphi_{\jj})\\
	&=Cq^nD(\pi_V \varphi_{\ii},\pi_V \varphi_{\jj})\\
        &\geq C(qc^\ell)^n,
	\end{align*}
	which completes the proof of the proposition.
\end{proof}
Since $\dim\eta^*>0$ (see \cite[Chapter VI.4]{BougerolLacroix1985}), we obtain the following corollary.
\begin{corollary}
	\label{cor:sep}Assume $\Phi=\{\varphi_{i}\}$ is exponentially separated, and $\{\overline{A}_i\}$ generate a non-compact and  strongly irreducible subgroup of $GL_2(\R)$.
	
	Then there exist $C>1$ and $\mathcal{Y}\subset\mathbb{RP}^{1}$,
	with $\eta^*(\mathcal{Y})=1$, so that for every $V\in\mathcal{Y}$ there exists $N=N(V)$ such that for each $n>N$ and 	$\ii,\jj\in\Xi_n^V$ with $\ii\neq\jj$
        $$
        \DD_{Cn}^{\AAA_{2,1}}(\pi_V\varphi_{\ii})\ne\DD_{Cn}^{\AAA_{2,1}}(\pi_V\varphi_{\jj}).
        $$
\end{corollary}

\begin{proof}
	Let $s=\dim_H\eta^*/2$. Let $E=E(s)$ be the set and let $c=c(s)$ be the constant defined in Proposition~\ref{prop:separation}, and set $\mathcal{Y}=\RP\setminus E$. Thus, for every $V\in\mathcal{Y}$ there exists $N=N(V)$ such that
	$$
	d(\pi_V\varphi_{\ii},\pi_V\varphi_{\jj})>c^n\text{ for every }n\geq N\text{ and }\ii\neq\jj\in\Xi_n^V.
	$$
	 By property \eqref{item:ball} of the partitions $\QQ_i$ (Section~\ref{sub:q-adic-partitions}), there exists $K>1$ such that $\DD_n(\psi)\subset B_{Kq^{-n}}(\psi)$ for all $\psi\in\AAA_{2,1}$ and $n$. Thus, by choosing $C=\frac{-\log c+\log K}{\log q}$, $\DD_{Cn}^{\AAA_{2,1}}(\pi_V\varphi_{\ii})\ne\DD_{Cn}^{\AAA_{2,1}}(\pi_V\varphi_{\jj})$ for every large enough $n$ and $\ii\neq\jj\in\Xi_n^V$.
	
	Moreover, $\dim_H\mathcal{E}\leq s$ and thus, by the definition of Hausdorff dimension of measures, $\eta^*(\mathcal{Y})=1$.
\end{proof}

%The only change required in Proposition \ref{prop:separation}
%in order to prove the same in $\mathbb{R}^{d}$ is to work with $d+1$
%affinely independent points instead of the three points $x,y,z$.
%Note that, in $\mathbb{R}^{d}$ as well well as in $\mathbb{R}^{2}$,
%the set of $\pi\in O_{d,1}$ such that $d(\pi x,\pi x')<\delta$ is
%of diameter $O(\delta/d(x,x'))$.

\section{\label{sec:main-theorem}Proof of the main theorems}

\begin{comment}
We shall derive at a contradiction with \eqref{eq:11} by using Theorem
	\ref{thm:inverse-theorem-for-projections}. For this we shall need
	the following lemmas. For $\mathcal{N}\subset\mathbb{N}$ we write,
	\[
	\underline{d}(\mathcal{N})=\underset{k}{\liminf}\:\frac{\#\{n\in\mathcal{N}\::\:n\le k\}}{k}\:.
	\]
\end{comment}

\subsection{Proof of Theorem \ref{thm:one-dimensional-projections}}
In this section we prove Theorem \ref{thm:one-dimensional-projections}, whose statement We recall for convenience:

\begin{thm*}
Let $\mu=\sum p_{i}\cdot\varphi_{i}\mu_{i}$
be a self-affine measure in $\mathbb{R}^{2}$ such that $\{\varphi_{i}\}$
has exponential separation. Suppose that the normalized linear parts
of $\varphi_{i}$ generate a totally irreducible, non-compact subgroup
of $GL_{2}(\mathbb{R})$. Then for $\eta^*$-a.e. $V\in\RP$, we have $\dim\pi_{V}\mu=\min\{1,\dim\mu\}$.
\end{thm*}

\begin{comment}
  For $\Xi\subseteq\Lambda^*$ such that $\{[\ii]\}_{\ii\in\Xi}$ is a partition of $\Lambda^\N$, and for $V\in\RP$,  let $\theta_{\Xi}^V$ be the measure on $\AAA_{2,1}$ such that
$$
\theta_{\Xi}^V=\sum_{\ii\in\Xi}p_{\ii}\delta_{\pi_{V}\circ\varphi_{\ii}}.
$$
By \eqref{eq:iterated-convolution-for-mu}, for every such partition $\Xi$ and $V\in\RP$ we have\[
\pi_V\mu=\theta_{\Xi}^V\boldsymbol{.}\mu.
\]
\begin{theorem}\label{thm:proj}
  Assume $\Phi=\{\varphi_{i}\}$ is exponentially separated and $\mathcal{G}$ is strongly irreducible and not contained in any compact subgroup of $GL_2(\R)$. Then for
	$\eta^*$-a.e. $V\in\mathbb{RP}^{1}$ we have $\dim_H\pi_V\mu=\min\{1,\dfrac{h_{\m}}{-\chi_1}\}$.
\end{theorem}
\end{comment}

We begin the proof. Write $\beta=\min\{1,\dfrac{H(p)}{{|\chi_1|}}\}$ and let $\alpha$ be with $\dim\pi_V\mu=\alpha$
	for $\eta^*$-a.e. $V\in\mathbb{RP}^{1}$, as in Lemma~\ref{lem:alpha}. Assume by way of contradiction	that $\alpha<\beta$ (it is enough to show that $\alpha=\beta$  because $\dim\pi_V\mu\leq \dim\mu\leq H(p)/|\chi_1|$).
	
	Let $C$ and $\mathcal{Y}$ be as obtained in
	Corollary \ref{cor:sep}, and let $V\in\mathcal{Y}$ with 
	\[
	\lim_{n\to\infty}\frac{1}{n}H(\pi_V\mu,\DD_{n})=\alpha.
	\]
	Since  conditional entropy is the difference of the entropies (eq. \ref{eq:conditional-entropy-is-increment}),
	\begin{equation}
	\lim_{n\to\infty}\frac{1}{Cn}H(\pi_V\mu,\DD_{(C+1)n}|\DD_{n})=\alpha.\label{eq:1}
	\end{equation}
	Write $\theta_n=\theta_n^V$ for the discrete measure on $\AAA_{2,1}$ given by \[
        \theta_n=\sum_{\ii\in\Xi_n^V}p_{\ii}\delta_{\pi_{V}\circ\varphi_{\ii}}
        \]
        (we suppress the dependence on $V$ when it is not needed). By the definition of $\Xi_n^V$ and \eqref{eq:iterated-convolution-for-mu}, we have \[
        \pi_V\mu=\theta_n^V\conv \mu.
        \]
        Since $\theta_{n}=\mathbb{E}_{i=0}((\theta_{n}^V)_{\psi,i})$, also
	$\theta_{n}\boldsymbol{.}\mu=\mathbb{E}_{i=0}((\theta_{n}^V)_{\psi,i}\boldsymbol{.}\mu)$,
	and so by \eqref{eq:1} and by the concavity of conditional entropy,
	\[
	\limsup_{n\to\infty}\:\mathbb{E}_{i=0}\left(\frac{1}{Cn}H((\theta_{n}^V)_{\psi,i}\boldsymbol{.}\mu,\DD_{(C+1)n}|\DD_{n})\right)\leq\alpha\:.
	\]
	For $n\ge1$ and $\psi\in\supp(\theta_{n}^V)$ we have $\Vert\psi\Vert=\Theta(q^{-n})$,
	hence
	\[
	\diam(\supp((\theta_{n}^V)_{\psi,0}\boldsymbol{.}\mu))=O(q^{-n}),
	\]
	and so the conditioning can be removed, since it results in an $O(1)$ change to entropy, which is killed by the denominator $n$:
	\begin{equation}
	\limsup_{n\to\infty}\:\mathbb{E}_{i=0}\left(\frac{1}{Cn}H((\theta_{n}^V)_{\psi,i}\boldsymbol{.}\mu,\DD_{(C+1)n})\right)\leq\alpha\:.\label{eq:11}
	\end{equation}

	\begin{lemma} With $C$ and $V\in\mathcal{Y}$ as above,

		\label{lem:ent-theta_n-D_Cn}
		$$
		\lim_{n\to\infty}\frac{1}{n}H(\theta_n^V,\DD_{Cn}) = \frac{H(p)}{{|\chi_1|}}.
		$$
	\end{lemma}
	
	\begin{proof}
	By choice of $V$, for every $n$ large enough  $\DD_{Cn}(\psi_{1})\ne\DD_{Cn}(\psi_{2})$
		for distinct $\psi_{1},\psi_{2}\in\supp\{\theta_{n}^V\}$. Moreover, since $\theta^V_n$ is the distribution of $\pi_V\varphi_{\JJ (n,V)}$, where $\JJ (n,V)$ is a random sequence from $\Xi_n^V$ (see \ref{sub:random-cylinder-measures}), we have \begin{eqnarray*}
                  \frac{1}{n}H(\theta^V_n,\QQ_{Cn}) & = & - \EE(\frac{1}{n}\log p_{\JJ (n,V)})\\
                 & = & - \EE(\frac{|\JJ (n,V)|}{n}\cdot\frac{1}{|\JJ (n,V)|}\log p_{\JJ (n,V)})
                \end{eqnarray*}
                By Proposition \ref{prop:furstenberg-measure} and the definition of $\JJ (n,V)$, we know that $|\JJ (n,V)|/n\to 1/|\chi_1|$, and by Shannon-McMillan-Breiman we know that \[
                p_{|\JJ (n,V)|}=q^{-(H(p)+o(1))|\JJ (n,V)|}\;\;\textrm{as }n\to\infty.
                \]
                Furthermore, since $p_i$ and $\|\varphi_i\|$ are uniformly bounded away from $0,\infty$, both terms in the expectation above are bounded. The lemma follows  by bounded convergence.

	\end{proof}

	\begin{lemma}
		\label{lem:ent_psi.mu}Let $C$ be as above and $\tau>0$. Then for $\eta^*$-a.e. $V\in\RP$,
		\[
		\limsup_{n\to\infty}\theta_{n}^V(\{\psi\::\:\frac{1}{Cn}H(\psi\mu,\DD_{(C+1)n})>\alpha-\tau\})=1.
		\]
	\end{lemma}
	
	\begin{proof}
          Recalling the definition of $\theta_n^V$ again, what we need to show is \[
          \limsup_{n\to\infty}\PP(\frac{1}{Cn}H(\pi_{V}\varphi_{\JJ(n,V)}\mu,\QQ_{(C+1)n})>\alpha-\tau) = 1.
          \]
          The affine map $\pi_V\varphi_{\JJ (n,V)}$ differs from the projection $\pi_{A^*_{\JJ (n,V)V}}$ by a translating and scaling by $\|A^*_{\JJ (n,V)}|V\|$. Doing the same to the partition $\QQ_{(C+1)n}$, the entropy incurs an error of $O(1/n)$, which we can absorb in $\tau$. With this transformation, the claim becomes\[
          \limsup_{n\to\infty}\PP(\frac{1}{Cn}H(\pi_{A^*_{\JJ (n,V)}V}\mu,\QQ_{((C+1)n-\log \|A^*_{\JJ (n,V)}|V\|)})>\alpha-\tau) = 1.
          \]
          Using $\|A^*_{\JJ (n,V)}|V\|=q^{-(1+o(1))n}$, the partition above differs from $\QQ_{Cn}$ by $o(n)$ scales, which incurs an $o(n)$ error in entropy, and after dividing by $n$ this can again be absorbed in $\tau$, so we in fact must show \[
          \limsup_{n\to\infty}\PP(\frac{1}{Cn}H(\pi_{A^*_{\JJ (n,V)}V}\mu,\QQ_{Cn})>\alpha-\tau) = 1.
          \]
          The last limsup will follow if we show that \[
          \limsup_{n\to\infty}\frac{1}{N}\sum_{n=1}^N \PP( \frac{1}{Cn}H(\pi_{A^*_{\JJ (n,V)}V}\mu,\QQ_{Cn})>\alpha-\tau) = 1.    
          \]
          By Lemma \ref{lem:converting-Un-to-In-or-Jn}, this follows from the same expression with $\UU(n)$ instead of $\\JJ (n,V)$, that is, from \[
          \limsup_{N\to\infty}\frac{1}{N}\sum_{n=1}^N\PP(\frac{1}{Cn}H(\pi_{A^*_{\II(n)V}}\mu,\QQ_{Cn})>\alpha-\tau) = 1.          
          \]
          (in Lemma \ref{lem:converting-Un-to-In-or-Jn},  $\PP_{1\leq n \leq N}(\ldots)$ appears instead of the explicit average over scales). We now note that, writing $f_n$ for the indicator function of the event in the last equation, the functions $f_n$ converge boundedly $\eta^*$-a.e. to $1$, and the average above is an integral of ergodic average of the form $(1/N)\sum_{n=1}^N f_n(T^nx)$, the transformation being the one underlying the random walk. This converges to $1$ by Maker's ergodic theorem.
          
	\end{proof}

                 Recall that $\nu=p^\N$ is the product measure on $\Lambda^\N$, that  $\Pi:\Lambda^\N\to\R^2$ is the coding map, and that $\mu=\Pi\nu$ (Section \ref{sub:self-affine-measures}). Also, given $V\in\RP$ let $\Pi_n=\Pi^V_{n}:\Lambda^{\N}\rightarrow\mathbb{R}^{2}$ be the map
	\[
	\Pi_{n}(\ii)=\varphi_{\Xi_n^V(\ii)}(0).
	\]
        We suppress $V$ in the notation when it is fixed. The choice to evaluate the maps at $0$ is arbitrary.
        \begin{lemma}
          For $\eta^*$-a.e. $V\in\RP$,
          \[
            \frac{1}{n}H(\pi_V\Pi_{n}\m,\DD_{n}) = \alpha + o(1) \;\;\textrm{as } n\to\infty.
          \]
        \end{lemma}  
         \begin{proof}
           Observe that
		\[
		|\pi_V\Pi(\ii)-\pi_V\Pi_{n}(\ii)|=O(q^{-n})\text{ for all }\ii\in\Lambda^{\N},
		\]
		hence the identity $\mu=\Pi\nu$, together with equation \eqref{eq:entropy-under-transformation-3}, give
		\begin{eqnarray*}
		  H(\pi_V\Pi_n\nu,\DD_{n}) & = & H(\pi_V\Pi\nu,\DD_n)+O(1)\\
                  & = & H(\pi_V\mu,\DD_{n})+O(1)\:.
		\end{eqnarray*}
		Since an $\eta^*$-typical $V$ satisfies $\frac{1}{n}H(\pi\mu,\DD_{n})\overset{n}{\rightarrow}\alpha$, the lemma is proved
        \end{proof}

        \begin{lemma} \label{lem:ent-theta_n-D_0}For $\eta^*$-a.e. every $V\in\RP$, we have
		\[
		\underset{n\to\infty}{\lim}\:\frac{1}{n}H(\theta_{n}^V,\DD_{0})=\alpha\:.
		\]
	\end{lemma}

        \begin{remark}
          This lemma is strongly sensitive to the metric on $\AAA_{2,1}$. If we were working with the norm-induced metric, we would use $\QQ_n$ instead of $\QQ_0$ as above, because the maps in the support of $\theta_n^V$ contract by $q^n$. But we are using a metric left-invariant by scaling, and as the proof shows, the correct scale is $\QQ_0$. 
        \end{remark}
        
	\begin{proof}
		Recall that dilation by $c$ is denoted $S_c$ and translation by $s$ is denoted $T_s$. Let $n\ge1$ and $V\in\RP$. 
		Consider the evaluation map $f:\AAA_{2,1}\rightarrow\mathbb{R}$ given by $f(\psi)=\psi(0)$. Since $\Vert\psi\Vert=\Theta(1)$ for $\psi\in\supp(S_{q^{n}}\theta_{n}^V)$, the linear part of $\psi\in S_{q^n}\supp\theta_n^V$ is contained in a compact subset $K\subseteq \AAA_{2,1}$ which intersects $O(1)$ level-$0$ $q$-adic cells in $\AAA_{2,1}$. Furthermore, every $\psi\in\supp S_{q^n}\theta_n^V$ can be represented uniquely as $T_{\psi(0)}\pi$ where $\pi\in K$. It follows that for every $Q\in\QQ_0^\R$, the pre-image $f^{-1}(Q)\cap\supp S_{q^n}\theta_n^V$ is contained in $\cup_{s\in Q} T_sK$. Since $\{T_s\}$ act isometrically on $\AAA_{2,1}$, the diameter  of this set is uniformly bounded in $Q$ (because one easily shows that the diameter of $\{T_s\psi\,:\,s\in [a,b]\}$ is $O(b-a)$ for any interval $[a,b]$ and $\psi\in\AAA_{2,1}$), and it follows that each $f^{-1}Q\cap\supp S_{q^n}\theta_n^V$ intersects $O(1)$ atoms of $\QQ_0^{\AAA_{2,1}}$. Also, it is clear that there exists an open set $U\subseteq \AAA_{2,1}$ such that  $f^{-1}Q$ contains $T_sU$ for some $s\in\R$, which shows that each atom of $\QQ_0$ intersects boundedly many sets of the form $f^{-1}Q$. In short, on the supports of $S_{q^n}\theta_n^V$, the partitions $\QQ_0$ and $f^{-1}\QQ_0$ are commensurable. Therefore, \begin{eqnarray*}
                H(S_{q^n}\theta_n^V,\QQ_0) & =& H(S_{q^n}\theta_n^V,f^{-1}(\QQ_0)) + O(1)\\
		  & = &  H(fS_{q^{n}}\theta_{n}^V,\DD_{0})+O(1)\\
                  & = & H(f\theta_{n}^V,\DD_{n})+O(1)                  
                \end{eqnarray*}
                where in the last transition we used the fact that $f$ and $S_q^n$ commute, and that entropy of a scaled measure in $\R$ should be measured in a correspondingly smaller level. Now,                since we are working with the invariant metric, the left hand side of the last equation is just $H(\theta_{n}^V,\DD_{0}) + O(1)$, the quantity we are interested in. As for the right hand side, we have $f\theta_n^V = \pi_V\Pi_n\nu$, where $\Pi_n=\Pi_n^V$ is  as defined before the previous lemma, so by that lemma, after dividing the last equation by $n$, we obtain the desired result for $\eta^*$-a.e. choice of $V$.
	\end{proof}

	We now return to the main proof, with $C,\mathcal{Y}$ as before. Let $\tau>0$ be small, it follows from Lemma~\ref{lem:ent-theta_n-D_Cn} and Lemma \ref{lem:ent-theta_n-D_0} that for any $n$ large enough
	\begin{eqnarray*}
		\mathbb{E}_{i=0}\left(\frac{1}{Cn}H((\theta_{n}^V)_{\psi,i},\DD_{Cn})\right) & = & \frac{1}{Cn}H(\theta_{n}^V,\DD_{Cn}\mid\DD_{0})\\
		& = & \frac{1}{Cn}H(\theta_{n}^V,\DD_{Cn})-\frac{1}{Cn}H(\theta_{n}^V,\DD_{0})\\
		& > & \frac{\beta-\alpha-\tau}{C}.
	\end{eqnarray*}
	We may assume $\tau$ is so small such that $\beta-\alpha-\tau>0$,
	hence there exists $\varepsilon>0$ such that for every $n$
	large enough,
	\begin{equation}
	\mathbb{P}_{i=0}\left(\frac{1}{Cn}H((\theta_{n}^V)_{\psi,i},\DD_{Cn})>\varepsilon\right)>\varepsilon\:.\label{eq:14a}
	\end{equation}
	Let $\delta=\delta(\varepsilon)>0$ be as obtained in Theorem \ref{thm:inverse-theorem-for-projections}.
	From this theorem, the inequality \eqref{eq:14a}, and Lemma~\ref{lem:ent_psi.mu}, it follows that for every $N$ large enough there exists $n\geq N$
	\[
	\mathbb{E}_{i=0}\left(\frac{1}{Cn}H((\theta_{n}^V)_{\psi,i}\boldsymbol{.}\mu,\DD_{(C+1)n})\right)\ge\alpha+\varepsilon\cdot\delta-O(\tau)\:.
	\]
	Since $\tau$ can be taken to be arbitrarily small with respect to
	$\varepsilon$ and $\delta$, we have thus reached a contradiction with
	\eqref{eq:11}. Hence we must have $\alpha=\beta$, which completes
	the proof of Theorem \ref{thm:one-dimensional-projections}.

\subsection{Proof of Theorems \ref{thm:main-sets} and \ref{thm:main-measures}}

The result on the Hausdorff dimension of the self-affine set follows from the corresponding statement for self-affine measures and \cite[Proposition~4.1]{MorrisShmerkin2016}.

The result on self-affine measures follows by \cite[Corollary~2.7, Corollary~2.8]{BaranyKaenmaki2015} from Theorem~\ref{thm:one-dimensional-projections}, and the fact that the SOSC implies exponential separation. Let us show this. Let $U$ be  as in the definition of the SOSC.  Let $x\in X\cap U$, so that $d(x,\partial U)>0$. Therefore, for every
$\ii\in\Lambda^*$, $d(\varphi_{\ii}(x),\partial \varphi_\ii U)\geq \alpha^{|\ii|}d(x,\partial U)$, where
$0<\alpha=\min_i\|A_i^{-1}\|^{-1}<1$. Now let $\ii\jj\in\Lambda^n$, and suppose that $\ii\neq\jj$. Let $\uu\in\Lambda^k$ be a maximal common initial segment of $\ii,\jj$, so that $\ii=\uu\ii'$ and $\jj=\uu\jj'$ and $\ii',\jj'$ differ in their first symbol. Then 
\begin{eqnarray*}
d(\varphi_{\ii}(x), \varphi_{\jj}(x)) & = & d(\varphi_\uu(\varphi_{\ii'}(x)),\varphi_\uu(\varphi_{\jj'}(x))) \\
& \geq & \alpha^k d(\varphi_{\ii'}(x),\varphi_{\jj'}(x))\\
& \geq & \alpha^k d(\varphi_{\ii}(x),\partial \varphi_\ii U)\\
& \geq & \alpha^nd(x,\partial U).
\end{eqnarray*}
This is exponential separation.

\subsection{\label{sub:triangular-case}The triangular case}

Our previous results imposed conditions on the matrices $A_i$ which ensure that the Furstenberg measure $\eta^*$ is unique. More than uniqueness itself, we used the implication that the random sequence $A_{\UU(n)}V$ converge in distribution to a measure $\eta^*$ for every initial $V\in\RP$, and  that $\eta^*$ has positive dimension. However, if ``every $V$'' is relaxed to ``all but one'', the entire argument carries through unchanged (using the fact that the exceptional $V$ has $\eta^*$-measure zero). This observation enables us to treat another important class of self-affine sets and measures.

\begin{proposition}\label{prop:triang}
	Let $\Psi$ be a self-affine IFS on the plane in the form
	$$
	\Psi=\left\{\varphi_i(x)=\left(\begin{matrix}
	a_i & 0 \\
	b_i & c_i
	\end{matrix}\right)x+\left(\begin{matrix}
	u_i \\
	v_i 
	\end{matrix}\right)\right\}_{i\in\Lambda}.
	$$
	Suppose that $a_i<c_i$ for every $i\in\Lambda$, $\Psi$ satisfies the strong open set condition and the matrices are not simultaneously diagonalizable. Then for every self-affine measure $\mu=\sum_{i\in\Lambda}p_i\cdot \varphi_i\mu$,
	$$
	\\dim\mu=\lydim\nu,
	$$
	and in particular $\dim X=\adim\Psi$. Moreover, for every $V\in\RP$ except the span of $(0,1)$,
	$$
	\dim\pi_V\mu=\min\{1,\dim\mu\},
	$$
	and $\dim\pi_VX=\min\{1,\dim X\}$.
\end{proposition}

\begin{proof}
  Let $V$ denote the horizontal line through the origin (spanned by $(1,0)$). This is a common eigenvector for the $A^*_i$, where $A_i$ is the linear part of $\varphi_i$. Parametrizing all lines other than $V$ by the representatives $(z,1)$, we see that the action of $A_i^*$ on  $\RP\setminus\{V\}$ is given by $z\mapsto \frac{a_iz+b_i}{c_i}$. Considered as maps of the real variable $z$, these maps are strictly contracting due to the hypothesis $a_i<c_i$, and so they define a self-similar measure $\eta^* = \sum_{i\in\Lambda}p_I\cdot A_i^*\eta^*$ (we use the same notation for the measure on $\RP$ and on the line). Since $A_i^*$ were assumed to not be jointly diagonalizable, they do not share a common fixed point in $\RP$ other than $V$, hence no common fixed point in $\RP\setminus\{V\}$. This implies that $\dim\eta^*>0$. Furthermore, for any $z\in\RP\setminus\{V\}$ the random walk started from $z$ converges to $\eta^*$. 	
Thus, the proof of the theorem follows in the same way as the proof of   Theorem~\ref{thm:main-measures}, as discussed above.

The assertions on the orthogonal projections follow from Theorem \ref{thm:mainsep2} in the next section.
\end{proof}

\section{\label{sec:applications}Applications}

In this section, we present some applications of our main theorems: to the orthogonal projection of self-affine measures along every direction, dimension of the graph of the Takagi function, and the dimension of attractors of skew-product Horseshoes. 

\subsection{Dimension of every projection} The main theorem of this section is the following.

\begin{theorem}\label{thm:mainsep2}
	If $\Phi=\{\varphi_{i}\}_{i\in\Lambda}$ satisfies SOSC and $\{A_i\}_{i\in\Lambda}$ is strongly irreducible then for every self-affine measure  $\mu$ and every $V\in\RP$
	$$
	\dim_H\pi_V\mu=\min\{1,\dim_H\mu\}.
	$$
	Moreover, for the attractor $X$ of $\Phi$
	$$
	\dim_H\pi_VX=\min\{1,\dim_H X\}.
	$$
\end{theorem}

The proof is a slight modification of the result of Falconer and Kempton \cite{FalconerKempton2017}. For the convenience of the reader, we give sketch of the proof with the required modifications. First, we state the variant of the result Hochman
and Shmerkin \cite{HochmanShmerkin}, which is suitable for the projections of self-affine systems. 

Throughout this section we always assume that $\{\bar{A}_i\}_{i\in\Lambda}$ is strongly irreducible and the generated subgroup is unbounded.

\begin{theorem}\label{thm:lowerbound}
Let $V\in\RP$. If for $\m$-a.e. $\ii\in\Lambda^\N$
$$
\lim_{N\to\infty}\liminf_{n\to\infty}\frac{1}{Nn}\sum_{k=1}^nH(\pi_V\varphi_{\Xi_{Nk}^V(\ii)}\mu,\DD_{N(k+1)})\geq \beta,
$$
then $\dim_H\pi_V\mu\geq\beta$.
\end{theorem}

The proof of the theorem can be found in \cite[Theorem~4.3]{FalconerKempton2017}.

Let
$$
\Gamma=\left\{(\ii,V,t)\in\Lambda^\N\times\RP\times\R:0\leq t\leq-\log\|A_{i_0}^*|V\|\right\}.
$$
We identify the points $(\ii,V,-\log\|A_{i_0}^*|V\|)$ in $\Gamma$ with $(\sigma\ii,A_{i_0}^*V,0)$. By using the identification, we can extend the suspension flow 
$$
F_s(\ii,V,t)=(\ii,V,t+s)\text{ for }s\leq-\log\|A_{i_0}^*|V\|-t
$$ 
to every $s>0$. It is easy to see that with respect to the normalized measure $(\m\times\eta^*\times\mathcal{L})_\Gamma$, the flow $F_s$ is invariant and ergodic. For simplicity, denote $\lambda$ the measure $(\m\times\eta^*\times\mathcal{L})_\Gamma$.

\begin{lemma}\label{lem:erg}
For every $N\geq1$, the time-$N$ map $F_N\colon\Gamma\mapsto\Gamma$ is ergodic w.r.t the measure $\lambda$.
\end{lemma}

For the proof of the lemma, we refer to \cite[Lemma~5.3]{FalconerKempton2017}. 

\begin{lemma}\label{lem:finite}
For every $V,V'\in\RP$,
$$
\m\left(\left\{\ii\in\Lambda^*:\sup_{n\geq1}\frac{\left|\log\|A_{\ii|_n}^*|V\|-\log\|A_{\ii|_n}^*|V'\|\right|}{d_\RP(V,V')}<\infty\right\}\right)=1.
$$
\end{lemma}

\begin{proof}
Since the map $V\mapsto\log\|A|V\|$ is Lipschitz for any matrix $A$, we have
\begin{align*}
\left|\log\|A_{\ii|_n}^*|V\|-\log\|A_{\ii|_n}^*|V'\|\right|&\leq\sum_{k=1}^n\left|\log\|A_{i_k}^*|A_{\ii|_{k-1}}^*V\|-\log\|A_{i_k}^*|A_{\ii|_{k-1}}^*V'\|\right|\\
&\leq C\sum_{k=1}^nd_\RP(A_{\ii|_{k-1}}^*V,A_{\ii|_{k-1}}^*V')\\
&\leq Cd_{\RP}(V,V')\sum_{k=1}^n\frac{\alpha_2(A_{\ii|_{k-1}}^*)}{\alpha_1(A_{\ii|_{k-1}}^*)}\cdot\frac{\|A_{\ii|_{k-1}}^*\|^2}{\|A_{\ii|_{k-1}}^*|V\|\|A_{\ii|_{k-1}}^*|V'\|}.
\end{align*}
Thus, the statement follows by \cite[III.3.2]{BougerolLacroix1985} and the Oseledets's Theorem.
\end{proof}

For simplicity, let $R_N(\ii,V,t)=\frac{1}{N}H(\pi_V\mu,\DD_N)$. Moreover, we extend the definition of $\Xi_n^{V,q}$ to real valued $n$. That is 
$$
\Xi_s^V=\{\ii=(i_0,\ldots,i_m)\in\Lambda^*:\|A_{(i_0,\ldots,i_m)}^*|V\|\leq q^{-s}<\|A_{(i_0,\ldots,i_{m-1})}^*|V\|\}.
$$

\begin{lemma}\label{lem:generate}
For every $V'\in\RP$, and for every $\varepsilon>0$ there exists a set $\Omega_{V',\varepsilon}$ and a constant $C=C(\varepsilon)$ such that $\m(\Omega_{V',\varepsilon})>1-\varepsilon$ and for every $\ii\in\Omega_{V',\varepsilon}$
$$
\lim_{n\to\infty}\frac{1}{n}\sum_{k=1}^nR_N(F_{Nk}(\ii,V',0))=\int R_N(\ii,V,t)d\lambda(\ii,V,t)+O(\frac{C}{N}).
$$
\end{lemma}

\begin{proof}
Let us fix $V'\in\RP$. Then by Lemma~\ref{lem:finite},
$$
\Omega_{V'}^1=\left\{(\ii,V)\in\Lambda^*\times\RP:\sup_{n\geq1}\frac{\left|\log\|A_{\ii|_n}^*|V\|-\log\|A_{\ii|_n}^*|V'\|\right|}{d_\RP(V,V')}<\infty\right\}
$$
has full measure w.r.t $\m\times\eta^*$. For every $\varepsilon>0$, one can choose $C_1=C_1(\varepsilon,V')>0$ such that $\sup_{n\geq1}\frac{\left|\log\|A_{\ii|_n}^*|V\|-\log\|A_{\ii|_n}^*|V'\|\right|}{d_\RP(V,V')}<C$ with measure at least $1-\varepsilon/2$.
By \cite[Chapter III.3.2]{BougerolLacroix1985}, we can choose $C_2=C_2(\varepsilon,V')>0$ such that
$$
\Omega_{V'}^2=\left\{(\ii,V)\in\Lambda^*\times\RP:\sup_{n\geq1}\frac{\|A_{\ii|_n}^*|V\|}{\|A_{\ii|_n}^*|V'\|}<C_2\right\},
$$
and $\m\times\eta^*(\Omega_{V'}^2)>1-\varepsilon/2$. Let $\Omega_{V'}=\Omega_{V'}^1\cap\Omega_{V'}^2$ and $C=\max\{C_1,C_2\}$.
 
Hence, for every $(\ii,V)\in\Omega_{V'}$
\begin{align*}
\frac{1}{N}H(\pi_{A_{\Xi_{Nk+t}^V(\ii)}^*V}\mu,\DD_{N})&=\frac{1}{N}H(\pi_{V}\varphi_{\Xi_{Nk+t}^V(\ii)}\mu,\DD_{N(k+1)+t})+O(\frac{1}{N})\\
&=\frac{1}{N}H(\pi_{V}\varphi_{\Xi_{Nk}^{V'}(\ii)}\mu,\DD_{N(k+1)})+O(\frac{1+\log C}{N})\\
&=\frac{1}{N}H(\pi_{A_{\Xi_{Nk}^{V'}(\ii)}^*V}\mu,\DD_{N})+O(\frac{1+\log C}{N})\\
&=\frac{1}{N}H(\pi_{A_{\Xi_{Nk}^{V'}(\ii)}^*V'}\mu,\DD_{N})+O(\frac{1+\log C+d_{\RP}(A_{\Xi_{Nk}^{V'}(\ii)}^*V,A_{\Xi_{Nk}^{V'}(\ii)}^*V')}{N}).
\end{align*}
By Birkhoff's ergodic theorem and Lemma~\ref{lem:erg},
$$
\{F_{Nk}(\ii,V,t)\}_{k=1}^{\infty}\text{ equidistributes for $\lambda$-a.e. $(\ii,V,t)\in\Omega_{V'}\times\R$}.
$$
On the other hand, by \cite[Theorem~III.4.3]{BougerolLacroix1985}, $d_\RP(A_{\ii|_n}^*V,A_{\ii|_n}^*V')\to0$ as $n\to\infty$ for $\m$-a.e. $\ii$. Thus, for every a.e. $(\ii,V)\in\Omega_{V'}$
\begin{align*}
\int R_Nd\lambda&=\lim_{n\to\infty}\frac{1}{n}\sum_{k=1}^nR_N(F_{Nk}(\ii,V,t))\\
&=\lim_{n\to\infty}\frac{1}{n}\sum_{k=1}^n\left(\frac{1}{N}H(\pi_{A_{\Xi_{Nk}^{V'}(\ii)^*}V'}\mu,\DD_{N})+\frac{d_{\RP}(A_{\Xi_{Nk}^{V'}(\ii)}^*V,A_{\Xi_{Nk}^{V'}(\ii)}^*V')}{N}\right)+O(\frac{1+\log C}{N}).
\end{align*}
\end{proof}

\begin{proof}[Proof of Theorem~\ref{thm:mainsep2}]
Let us denote the measure $\lambda\circ\pi_{\RP}$ by $\gamma^*$. It is easy to see that $\gamma^*$ is equivalent to $\eta^*$. Thus, by \cite[III.3.2]{BougerolLacroix1985} and Theorem~\ref{thm:one-dimensional-projections}, 
$$
\lim_{N\to\infty}\frac{1}{N}H(\pi_{V}\mu,\DD_{N})=\min\{1,\lydim\mu\}\text{ for $\gamma^*$-a.e. $V$}.
$$

Let $V'\in\RP$ be arbitrary but fixed. Let $\varepsilon>0$ and let $C>0$ and $\Omega_{V'}$ as in Lemma~\ref{lem:generate}. Then for every $\ii\in\Omega_{V'}$, Fatou's lemma
\begin{align*}
\lim_{N\to\infty}\liminf_{n\to\infty}\frac{1}{Nn}\sum_{k=1}^n&H(\pi_{V'}\varphi_{\Xi_{Nk}^{V'}(\ii)}\mu,\DD_{N(k+1)})\\
&=\lim_{N\to\infty}\liminf_{n\to\infty}\frac{1}{Nn}\sum_{k=1}^nH(\pi_{A_{\Xi_{Nk}^{V'}(\ii)}^*V'}\mu,\DD_{N})+O(\frac{1}{N})\\
&=\lim_{N\to\infty}\frac{1}{N}\int H(\pi_{V}\mu,\DD_{N})d\gamma^*(V)+O(\frac{1+C}{N})\\
&\geq\min\{1,\lydim\mu\}.
\end{align*}
Since $\varepsilon>0$ was arbitrary, the statement follows by Theorem~\ref{thm:lowerbound}.
\end{proof}

\subsection{Self-affine systems with reducible linear part}  

An important example for triangular self-affine set is the graph of the Takagi function, introduced by Takagi~\cite{takagi}. Namely, let $\phi(x)=d(x,\mathbb{Z})$ and $\lambda\in(1/2,1)$. Then the function $\Phi(x)=\sum_{n=0}^{\infty}\lambda^n\phi(2^n)$ is nowhere differentiable, $1$-periodic function. Let 
$$
\varphi_0(x,y)=\left(\begin{matrix}
\frac{x}{2} \\
\lambda y+\frac{x}{2}
\end{matrix}\right)\text{ and }\varphi_1(x,y)=\left(\begin{matrix}
\frac{x+1}{2} \\
\lambda y+\frac{1-x}{2}
\end{matrix}\right).
$$
Let $\Gamma$ be the graph of $\Phi$ over $[0,1]$, see Figure~\ref{fig:2}. It is easy to see that $\Gamma=\varphi_0(\Gamma)\cup\varphi_1(\Gamma)$. Thus,
\begin{equation}\label{eq:takagi}
\dim_H\Gamma\leq2+\frac{\log\lambda}{\log2}
\end{equation}
for every $\lambda\in(1/2,1)$.

\begin{figure}
	\centering
	\includegraphics[width=0.6\textwidth]{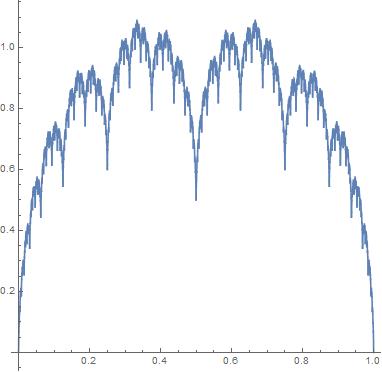}
	\caption{The graph of the function $\Phi$.}\label{fig:2}
\end{figure}
 
Ledrappier~\cite{ledgraph} gave a sufficient condition, which implies equality in \eqref{eq:takagi}. Using this condition, Solomyak~\cite{solomyakfamily} showed that equality holds in \eqref{eq:takagi} for Lebesgue almost every $\lambda\in(1/2,1)$. By applying Proposition~\ref{prop:triang}, we get the following corollary.

\begin{corollary}
	For {\bf every} $\lambda\in(1/2,1)$,
	$$
	\dim_H\Gamma=2+\frac{\log\lambda}{\log2}.
	$$
\end{corollary}

\subsection{Dimension of skew-product horseshoes}
In this section, we apply Theorem~\ref{prop:triang} for skew product horseshoe dynamical systems. That is, let $I=[0,1]$ and let $\{I_i\}_{i\in\Lambda}$ be a finite collection of disjoint closed intervals in $I$ such that $\diam(I_i)=1/\alpha_i<1$ and let the left-endpoint of $I_i$ be $t_i$. For a point $(x,y)\in I^2$ let
\begin{equation}\label{eq:F}
F(x,y)=\begin{cases}
         \left(\alpha_i(x-t_i),\beta_iy+\gamma_ix+\zeta_i\right), & \mbox{if } x\in I_i \\
         (-1,-1), & \mbox{otherwise},
       \end{cases}
\end{equation}
where $\beta_i<1$ for every $i\in\Lambda$. Without loss of generality, we can choose the parameters so that $I_i\times I$ is mapped into $I^2$ for every $i\in\Lambda$. For a visualization of the map $F$, see Figure~\ref{fig:1}. We not that $F$ is not invertible in general. 

\begin{figure}
	 \centering
		\includegraphics[width=0.8\textwidth]{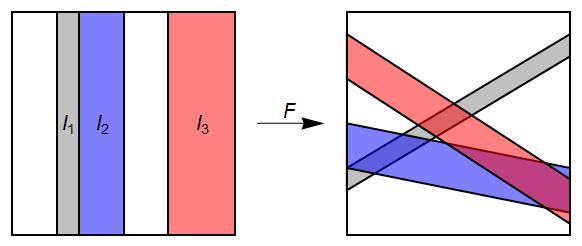}
	\caption{The first iterate of the function $F$.}\label{fig:1}
\end{figure}

Such systems have been already studied in several aspects. Falconer \cite{Falconerover} considered the dimension of the attarctors of a special family called "slanting Baker's transformation". Later, Simon~\cite{simon_1993}, Rams and Simon \cite{Rams_Simon} studied the non-linear generalizations of this type of systems in the aspect of dimension theory, and Tsujii \cite{Tsujii} and Rams \cite{Rams} studied in the aspect of absolute continuity. All these examples agree in the assumption that $\bigcup_{i\in\Lambda}I_i=I$. In the case of $\bigcup_{i\in\Lambda}I_i\subsetneq I$ there are only very few results are available, see for example Dysman \cite{Dysman}.

Skew-product Horseshoes are strongly related to hyperbolic planar endomorphisms with singularities, see Schmeling and Troubetzkoy \cite{schmeling1998dimension} and Persson \cite{Persson}.

The reason of the study of the system $F$ for us is that it is related to self-affine iterated function systems with reducible linear parts. Especially, the orthogonal projections of the attractors. Namely, the intersection of $\bigcap_{k=0}^{\infty}F^k(I^2)$ with vertical lines $\{x\}\times I$ corresponds to the projection of a self-affine set onto the subspace with slope $x$. Before we show this fact, we need several notations.

Observe that except a Cantor set, every point $(x,y)$ will eventually land at point $(-1,-1)$. So it is natural to restrict $F$ to this Cantor set. Namely, let
$Y$ be the self-similar set w.r.t the IFS $\{x\mapsto x/\alpha_i+t_i\}_{i\in\Lambda}$. Then for every $(x,y)\in Y\times[0,1]$, the map $F^n$ is defined for every $n\geq0$. Let $G=F|_{Y\times[0,1]}$. It is easy to see that $G$ is conjugated $(\Lambda^{\mathbb{Z}},\sigma)$. 

For simplicity, for $\ii\in\Lambda^{\mathbb{Z}}$, let us introduce the notation $\ii|_a^b=(i_a,\ldots,i_b)$ if $a\leq b$ and  $\ii|_a^b=\emptyset$ if $a>b$. We define the empty products as $1$ and the empty sums as $0$. So $G\circ\Pi=\Pi\circ\sigma$, where $\Pi(\ii)=(\Pi_1(\ii),\Pi_2(\ii))$ and
$$
\Pi_1(\ii)=\sum_{k=0}^{\infty}\dfrac{t_{i_k}}{\alpha_{\ii|_0^{k-1}}}\text{ and }\Pi_2(\ii)=\sum_{k=1}^{\infty}\beta_{\ii|^{-1}_{-(k-1)}}\left(\zeta_{i_{-k}}+\gamma_{i_{-k}}\Pi_1(\sigma^{-k}\ii)\right).
$$
Observe that $\Pi_1$ depends only on the positive coordinates of $\ii\in\Lambda^{\mathbb{Z}}$. Let us introduce two notations. Let
\begin{eqnarray*}
sl(\ii)&=&\sum_{k=1}^{\infty}\frac{\gamma_{i_{-k}}\beta_{\ii|^{-1}_{-(k-1)}}}{\alpha_{\ii|^{-1}_{-k}}}\\
tr(\ii)&=&\sum_{k=1}^{\infty}\beta_{\ii|^{-1}_{-(k-1)}}\left(\zeta_{i_{-k}}+\gamma_{i_{-k}}\sum_{n=1}^k\frac{t_{i_{-n}}}{\alpha_{\ii|_{-k}^{-n-1}}}\right).
\end{eqnarray*}
Observe that $sl(\ii)$ and $tr(\ii)$ depends only on the negative coordinates of $\ii$.
Simple algebraic manipulations show
\begin{align*}
\Pi_2(\ii)=\Pi_1(\ii)\cdot sl(\ii)+tr(\ii).
\end{align*}

for a $z\in\R$, let us introduce the projection $p_z(x,y)=xz+y$. We note that $p_z$ is bi-Lipschitz equivalent with the orthogonal projection onto the $1$ dimensional subspace containing $\binom{x}{1}$. 

Hence,  we can define a linear map $\omega^u(\ii_-):I\mapsto\R$ as
$$
\omega^u(\ii_-,x)=x\cdot sl(\ii_i)+tr(\ii_-).
$$

\begin{definition}
	We say that $\omega^u$ is a transversal family, if for any $\ii_-\neq\jj_-$ there exists $x\in[0,1]$ such that $\omega^u(\ii_-,x)=\omega^u(\jj_-,x)$ then $sl(\ii_-)\neq sl(\jj_-)$.
\end{definition}

\begin{proposition}\label{prop:induced}
	Let $F$ be as in \eqref{eq:F}. Suppose that $\alpha_i>1>\beta_i$ and $\omega^u$ is a transversal family. Let $\Psi$ be the self-affine IFS on the plane of the form
	$$
	\Psi=\left\{\varphi_i(x)=\left(\begin{matrix}
	\beta_i/\alpha_i & 0 \\
	\beta_it_i & \beta_i
	\end{matrix}\right)x+\left(\begin{matrix}
	\gamma_i/\alpha_i \\
	\zeta_i+\gamma_it_i 
	\end{matrix}\right)\right\}_{i\in\Lambda}.
	$$
	Let $X$ be its attractor and let $\widetilde{\Pi}$ be the natural projection from $\Lambda^\N$ to $X$. Then for $\Pi_1\nu$-a.e. $x\in Y$, $p_x\widetilde{\Pi}\nu=\mu_x$, and $p_x(X)=\{x\}\times[0,1]\cap\left(\bigcap_{k=0}^{\infty}F^k(I^2)\right)$ for every $x\in[0,1]$. Moreover, $\Psi$ satisfies the strong open set condition.
\end{proposition}

\begin{proof}
	We can identify $\Lambda^{\N}$ with $\Lambda^{\mathbb{Z}_-}$ and with a slight abuse of notation, we denote the natural projection from $\Lambda^{\mathbb{Z}_-}$ to $X$ by $\widetilde{\Pi}$. Denote $\widetilde{\Pi}_1$ the first and $\widetilde{\Pi}_2$ the second coordinate of $\widetilde{\Pi}$. Thus,
	\begin{eqnarray*}
		\widetilde{\Pi}_1(\ii)&=& sl(\ii)\\
		\widetilde{\Pi}_2(\ii)&=& tr(\ii).
	\end{eqnarray*}
	
	Let us observe that the corresponding Furstenberg measure generated by the transpose matrices, is equivalent to the self-similar measure generated by the IFS $\Phi=\{z\mapsto \frac{z}{\alpha_i}+t_i\}_{i\in\Lambda}$. Denote the natural projection w.r.t the IFS $\Phi$ by $\widetilde{\Pi}_F$. Thus, $	\widetilde{\Pi}_F(\ii_+)=\Pi_1(\ii)$ and
	$$
	\Pi_2(\ii)=\widetilde{\Pi}_1(\ii_-)\cdot\widetilde{\Pi}_F(\ii_+)+\widetilde{\Pi}_2(\ii_-),
	$$
	which gives the first claim of the proof. 
	
	By using the definition of $\omega^u(\ii_-)$, we get $sl(\ii_-)=\widetilde{\Pi}_1(\ii_-).$ So for any $\ii_-\neq \jj_-$, if there is no $x\in[0,1]$ so that $\omega^u(\ii_-,x)=\omega^u(\jj_-,x)$ then, in particular  $\widetilde{\Pi}_2(\ii_-)=\omega^u(\ii_-,0)\neq\omega^u(\jj_-,0)=\widetilde{\Pi}_2(\jj_-)$. If there exists such an $x\in[0,1]$ then $\widetilde{\Pi}_1(\ii_-)\neq\widetilde{\Pi}_1(\jj_-)$ by the transversality assumption, thus the strong open set condition holds.
\end{proof}

Denote the Hausdorff dimension of $Y$ by $s_{\alpha}$. That is, $\sum_{i\in\Lambda}\alpha_i^{-s_{\alpha}}=1$. Then $0<\mathcal{H}^s(Y)<\infty$. 

For a Bernoulli measure $\nu$ with probability vector $(p_i)_{i\in\Lambda}$, let $\chi_\alpha=\sum_{i\in\Lambda}p_i\log\alpha_i$ and $\chi_\beta=\sum_{i\in\Lambda}p_i\log\beta_i$.

Let $s_\beta$ be the unique solution of the equation $\sum_{i\in\Lambda}\beta_i^{s_\beta}=1$.

\begin{theorem}\label{thm:dynam}
Let $F$ be as in \eqref{eq:F}. Suppose that $\alpha_i>1>\beta_i$ for every $i\in\Lambda$ and $\omega^u$ is a transversal family. Then the following hold
\begin{enumerate}
	\item For every Bernoulli measure $\nu$ with probability vector , 
	$$
	\dim_H\Pi\nu=\frac{h_{\nu}}{\chi_{\alpha}}+\min\{1,\frac{h_{\nu}}{-\chi_{\beta}}\},
	$$
	\item for every $x\in[0,1]$, 
	$$
	\dim_H(\{x\}\times I)\cap\left(\bigcap_{k=0}^{\infty}F^k(I^2)\right)=\min\{1,s_{\beta}\}.
	$$ In particular,
	$$\dim_H\bigcap_{k=0}^{\infty}F^k(I^2)=1+\min\{1,s_{\beta}\}\text{ and  }\dim_H\bigcap_{k=0}^{\infty}G^k(Y\times I)=s_{\alpha}+\min\{1,s_{\beta}\}.$$
\end{enumerate}
\end{theorem}

\begin{proof}
	Let $\nu$ be a Bernoulli measure on $\Lambda^{\mathbb{Z}}$. Denote the conditional measure of $\Pi\nu$ on the foliation $\{x\}\times[0,1]$ by $\mu_{(x,y)}^s$. It is known that
	$$
	\dim_H\Pi\nu\geq\dim_H\mu^s_{(x,y)}+\dim_H\Pi_1\nu\text{ for $\Pi\nu$-a.e. $(x,y)$}.
	$$
	Thus, the lower bound in the first claim follows by Proposition~\ref{prop:induced} and Proposition~\ref{prop:triang}. The upper bound is straightforward.
	
	The second claim follow by choosing $\nu$ with the probability vector $(\beta_i^{s_{\beta}})_{i\in\Lambda}$, and applying Proposition~\ref{prop:induced} and Proposition~\ref{prop:triang}. The upper bounds for the dimension are straightforward.
\end{proof}

\bibliographystyle{abbrv}
\bibliography{Bibliography}

\end{document}